 \documentclass[reqno,11pt]{amsart}
 \usepackage{amsmath, amsthm, a4, latexsym, amssymb}
\usepackage[unicode]{hyperref}
\usepackage{srcltx}
\usepackage{xcolor}

\usepackage{bbm}

\usepackage{mathrsfs}

\def\@rmrk#1#2{\refstepcounter
    {#1}\@ifnextchar[{\@yrmrk{#1}{#2}}{\@xrmrk{#1}{#2}}}

 \renewcommand{\theequation}{\thesection.\arabic{equation}}
\makeatletter\@addtoreset{equation}{section}\makeatother

 \newfont{\bfit}{cmbxti10 scaled 1200}

\usepackage{amsfonts}
\usepackage{color}
\definecolor{Red}{rgb}{1,0,0}

\newtheorem{theorem}{Theorem}[section]
\newtheorem{proposition}[theorem]{Proposition}
\newtheorem{remark}[theorem]{Remark}

\newtheorem{lemma}[theorem]{Lemma}

\def \eps {\varepsilon} 
\def \d {\mathrm d}
\def \Q {\mathbb{Q}} 
\def \E {\mathbb{E}}
\def \P {\mathbb{P}}
\def \N  {\mathbb{N}}
\def \R  {\mathbb{R}}

\def \Z  {\mathbb{Z}}

\def \O {{\mathcal O}}
\def \e {{\mathrm e}}

\newcommand {\abs}[1] {\left\lvert {#1} \right\rvert}

\newcommand{\bbP}{\mathbb{P}}
\newcommand{\bbV}{\mathbb{V}}

\begin{document}
\baselineskip=14pt

\newcommand{\bfc}{\color{blue}} %% francis
\newcommand{\gfc}{\color{red}}
\newcommand{\ec}{\color{black}}

\newcommand{\heap}[2]{\genfrac{}{}{0pt}{}{#1}{#2}}
\newcommand{\sfrac}[2]{\mbox{$\frac{#1}{#2}$}}
\newcommand{\ssup}[1] {{\scriptscriptstyle{({#1}})}}

\renewcommand{\theequation}{\thesection .\arabic{equation}}
\numberwithin{equation}{section}

\renewcommand{\le}{\leq}

%Differentiation
\newcommand{\pa}{\partial}
\newcommand{\ffrac}[2]{{\textstyle\frac{{#1}}{{#2}}}}
\newcommand{\dif}[1]{\ffrac{\partial}{\partial{#1}}}
\newcommand{\diff}[1]{\ffrac{\partial^2}{{\partial{#1}}^2}}
\newcommand{\difif}[2]{\ffrac{\partial^2}{\partial{#1}\partial{#2}}}

\parindent=0cm

\title[Quenched and averaged large deviation rate functions for RWRE]{
The effect of disorder on quenched and averaged large deviations for random walks in random environments: boundary behavior}

%%%%%%%%%%%%%%%%%%%%%%%%%

\author{Rodrigo Bazaes, Chiranjib Mukherjee, Alejandro F. Ram\'\i rez and Santiago Saglietti}

%\thanks{ AMS 2000 {\it subject classifications}. Primary  60K35
%82C22, 82C41, 60K35, 82C22;
% secondary 82B23, 60B20
 %82C24, 60K05, 60G50, 60G70, 82B43.
% }

%\thanks{{\it Key words and phrases.} 
%Regeneration times,  Interacting Particle Systems, Front propagation.}
%Directed polymers, random medium, exactly solvable model, stable laws, product of random matrices}

%\thanks{$^2$ Pontificia Universidad Cat\'olica de Chile. Partially supported by Fondecyt grant 1171257}

%\thanks{$^3$ Pontificia Universidad Cat\'olica de Chile. Partially supported by Fondecyt grant XXX}

\address[Rodrigo Bazaes]{Facultad de Matem\'aticas\\
	Pontificia Universidad Cat\'olica de Chile\\
	Vicu\~na Mackenna 4860, Macul\\
	Santiago, Chile}
\email{{rebazaes@mat.uc.cl}}
\smallskip

\address[Alejandro F. Ram\'\i rez]{Facultad de Matem\'aticas\\
	Pontificia Universidad Cat\'olica de Chile\\
	Vicu\~na Mackenna 4860, Macul\\
	Santiago, Chile\\
	{and \\
		NYU-ECNU Institute of Mathematical Sciences at NYU Shanghai \\
		3663 Zhongshan Road North, Shanghai, 200062, China}}
\email{{aramirez@mat.uc.cl}}
\smallskip

\address[Chiranjib Mukherjee]{Fachbereich Mathematik und Informatik\\
	Universit\"at M\"unster\\
	Einsteinstrasse 62\\
	M\"unster 48149, Germany}

\email{chiranjib.mukherjee@uni-muenster.de}
\smallskip

\address[Santiago Saglietti]{Facultad de Matem\'aticas\\
	Pontificia Universidad Cat\'olica de Chile\\
	Vicu\~na Mackenna 4860, Macul\\
	Santiago, Chile\\
	{and \\Faculty of Industrial Engineering and Management\\
		Technion - Israel Institute of Technology\\
		Haifa 3200003, Israel}}

\email{sasaglietti@mat.uc.cl}

\begin{abstract} 
For a random walk in a uniformly elliptic and i.i.d. environment on $\Z^d$ with $d \geq 4$, we show that the quenched and annealed large deviations rate functions agree on any compact set contained in the boundary $\partial \mathbb{D}:=\{ x \in \R^d : |x|_1 =1\}$ of their domain which does not intersect any of the $(d-2)$-dimensional facets of $\partial \mathbb{D}$, provided that the disorder of the environment is~low~enough. As a consequence, we obtain a simple explicit formula for both rate functions on $\partial \mathbb{D}$ at low disorder. In contrast to previous works, our results do not assume any ballistic behavior of the random walk and are not restricted to neighborhoods of any given point (on the boundary $\partial \mathbb{D}$). In addition, our~results complement those in \cite{BMRS19}, where, using different methods, we investigate the equality of the rate functions in the interior of their domain. Finally,  for a general parametrized family of environments, we~show that the strength of disorder determines a phase transition in the equality of both rate functions, in the sense that for each $x \in \partial \mathbb{D}$ there exists $\varepsilon_x$ such that the two rate functions agree at $x$ when the disorder is smaller than $\varepsilon_x$ and disagree when its larger. This further reconfirms the idea, introduced in \cite{BMRS19}, that the disorder of the environment is in general intimately related with the equality of the rate functions.
\end{abstract}

%%%%%%%%%%%%%%%%%%%%%%%%%

\date{January 28, 2020}

%{started July 26, 2017}
\maketitle
%\tableofcontents
%%%%%%%%%%%%%%%%%%%%%%%%%%%%%%%%%%%%%%%%%%%%%%%%
%\bigskip

%  TO DO LIST: \begin{enumerate} \item  \end{enumerate}
%%%%%%%%%%%%%%%%%%%%%%%%%%%%%%%%%%%%%%%%%%%%%%%%%%%%%%%%%%%%%%%%%%%

\section{Introduction and background}

\noindent The model of a random walk in a random environment (RWRE) can be described as follows. Let $|x|_1$ denote the $\ell^1$-norm of any $x \in \R^d$ and define $\mathbb V:=\{ x \in \Z^d : |x|_1=1\}=\{\pm e_1,\dots,\pm e_d\}$, the set of all unit vectors in $\Z^d$, along with $\mathcal M_1(\mathbb V) := \big\{\vec{p}=(p(e))_{e \in \bbV} \in [0,1]^{\mathbb V}\colon \,\,  \sum_{e \in \mathbb V} p(e) = 1 \big\}$, the space of all probability vectors therein and the product space $\Omega:=(\mathcal M_1(\mathbb V))^{\Z^d}$ with the usual product topology. 
Any element $\omega \in \Omega$ will be called an \textit{environment}, i.e. each $\omega=(\omega(x))_{x \in \Z^d}$ is a sequence of probability vectors $\omega(x)=(\omega(x,e))_{e \in \mathbb V}$ on $\mathbb V$ indexed by the sites in the lattice. Given any $x\in \Z^d$ and $\omega \in \Omega$, the \textit{random walk in the environment $\omega$ starting at $x$ is defined as the Markov chain $(X_n)_{n \in \N_0}$ on $\Z^d$ whose law $P_{x,\omega}$ is given by 
\[
\begin{aligned}
&P_{x,\omega}(X_0=x)=1\quad\text{ and }\quad
&P_{x,\omega}(X_{n+1}=y+e\,|\,X_n=y)=\omega(y,e)\quad \forall\, y \in \Z^d\,,\,e \in \mathbb V.
\end{aligned}
\]
We call $P_{x,\omega}$ the {\it quenched law} of the RWRE. Then, if the environment $\omega$ is now chosen at random according to some Borel probability measure $\bbP$ on $\Omega$, we now obtain the measure $P_x$ on $\Omega \times (\Z^d)^{\N_0}$ defined as
\[
P_x( A \times B) := \int_A P_{x,\omega}(B)\mathrm{d}\P(\omega)\qquad \forall\, A \in \mathcal{B}(\Omega)\,,\,B \in \mathcal{B}((\Z^d)^{\N_0}).
\] We call $P_x$ the \textit{annealed law} of the RWRE and, in general, we will call the sequence $X=(X_n)_{n \in \N_0}$ under $P_x$ a RWRE with \textit{environmental law} $\bbP$.}
In the sequel, we shall work with environmental~laws satisfying the following assumption:

\noindent{\bf Assumption $\mathrm{A}$}: Under $\bbP$, the environment is {\it i.i.d.} (the random vectors $(\omega(x))_{x \in \Z^d}$ are independent and identically distributed) and \textit{uniformly elliptic}, i.e., there is a constant $\kappa>0$ such~that
\begin{equation}
\label{eq:unifk}
\bbP(\omega(x,e)\geq\kappa \text{ for all $x\in\mathbb Z^d$ and $e \in \bbV$})=1.
\end{equation}

In \cite{V03}, Varadhan proved that, for any $d\geq 1$ and under Assumption $\mathrm{A}$, both the quenched law $P_{0,\omega}\big(\frac {X_n}n\in \cdot\big)$ 
and its annealed version $P_{0}\big(\frac {X_n}n\in \cdot \big)$ satisfy a {\it large deviations principle} (LDP), i.e. there exist lower-semicontinuous functions $I_a,I_q:\mathbb R^d\to [0,\infty]$ such that for any $G\subset \mathbb R^d$ with interior $G^\circ$ and closure $\overline G$, 
\begin{align}
&-\inf_{x\in G^\circ}I_q(x)\le\liminf_{n\to\infty} \frac{1}{n}\log P_{0,\omega}\left(\frac{X_n}{n}\in G\right)\le\limsup_{n\to\infty}\frac{1}{n}\log P_{0,\omega}\left(\frac{X_n}{n}\in G\right)\le-\inf_{x\in\overline G} I_q(x) \label{VarQuenched}\\
& -\inf_{x\in G^\circ}I_a(x)\le\liminf_{n\to\infty} \frac{1}{n}\log P_{0}\left(\frac{X_n}{n}\in G\right)\le\limsup_{n\to\infty}\frac{1}{n}\log P_{0}\left(\frac{X_n}{n}\in G\right)\le-\inf_{x\in\overline G} I_a(x)\label{VarAnnealed}
\end{align}
with the first assertion being true for $\P$-almost every $\omega\in \Omega$. It can be shown that the rate functions $I_q$ and $I_a$ are both convex and are finite if and only if $x\in \mathbb D:=\{x\in\mathbb R^d: |x|_1\le 1\}$. Being also lower semicontinuous, the former implies that $I_q$ and $I_a$ are continuous on $\mathbb{D}$, see \cite[Theorem 10.2]{R97}. Moreover, by Jensen's inequality and Fatou's lemma, we always have the dominance $I_a(\cdot) \leq I_q(\cdot)$. 
In \cite{V03} it was also shown that, for $d \geq 2$, $I_a(0)=I_q(0)$ and both rate functions have the same zero-sets, leaving open the question of whether both rate functions are in fact equal in other parts of their domain. In this regard, Yilmaz showed later in \cite{Y11} that, for RWRE with $d \geq 4$ satisfying Assumption $\mathrm{A}$, both rate functions agree on \textit{some} neighborhood of the non-zero velocity, whenever the random walk satisfies Sznitman's \textit{condition}-(T) for ballisticity, see \cite{S01} for a precise definition.\footnote{In contrast, this has been shown to be false in \cite{YZ10} for dimensions $d \in \{2,3\}$: there exists a class of {\it non-nestling} random walks in i.i.d. and uniformly elliptic environments verifying that there is no neighborhood of the velocity on which the two rate functions are identical.} Recently in \cite{BMRS19}, we have shown that for $d \geq 4$ the two rate functions agree on any compact set in the interior of $\mathbb D$ which does not contain zero, provided that the disorder of the environment is low enough and regardless of whether the RWRE is ballistic. In the current work, we show that, despite the behavior of the RWRE on the boundary $\partial \mathbb D$ of $\mathbb D$ being quite different than in its interior, the above low-disorder phenomenon extends also to $\partial\mathbb D$. Indeed, we show that $I_q= I_a$ holds on any compact set contained in $\partial \mathbb{D}$ (avoiding its $(d-2)$-dimensional facets), provided that the disorder of the environment is sufficiently low. As a consequence, we obtain a simple explicit formula for the quenched rate function on $\partial \mathbb{D}$ at low disorder. %While the RWRE on $\partial\mathbb D$ bears some resemblance to random walks in {\it space-time (or dynamic) random environments}, in contrast to previous works our results are not restricted to neighborhoods of any specific point on~$\partial \mathbb{D}$ for sufficiently small disorder. 
Finally, for a general parametrized family of environments, we show that the strength of disorder determines a phase transition in the equality of both rate functions, in the sense that for each $x \in \partial \mathbb{D}$ there exists $\varepsilon_x$ such that the two rate functions agree at $x$ when the disorder is smaller than $\varepsilon_x$ and disagree when its larger. We turn to the precise statements of these results.

\section{Main result: Quenched and Annealed rate functions on the boundary}

Given any such environmental law $\mathbb{P}$, we define its {\it disorder} as
\begin{align}
&\mathrm{dis}(\mathbb{P}):= \inf\big\{ \varepsilon > 0 :  \xi(x,e) \in [1-\varepsilon, 1+\varepsilon],\,\, \bbP\text{-a.s. for all }e \in \mathbb{V} \text{ and }x \in \Z^d \big\},\label{eq:defdis} \\
&\qquad\mbox{with }\,\, \xi(x,e):=\frac{\omega(x,e)}{\alpha(e)} \quad\mbox{and}\quad \alpha(e):= \E[\omega(x,e)] \qquad\forall\, e \in \mathbb V. \label{def-q}, 
\end{align}
where $\E$ denotes expectation w.r.t. $\bbP$ and the definition of $\alpha(e)$ does not depend on $x\in\Z^d$ by Assumption $\mathrm{A}$. Moreover, both
$\xi(x,e)$ and $\mathrm{dis}(\bbP)$ are well-defined since $\bbP$ satisfies Assumption $\mathrm{A}$, whereas
$\mathrm{dis}(\mathbb{P})$ is the $L^\infty(\bbP)$-norm of the random vector $(\xi(x,e)-1)_{e \in \mathbb{V}}$ for any $x \in \Z^d$. 

We set $\partial{\mathbb D}=\{x\in \Z^d\colon |x|_1=1\}$ for the boundary of the unit ball and write 
 \begin{equation}\label{def-D}
\begin{aligned}
&{\partial\mathbb D(s):=\{x\in\mathbb R^d: |x|= 1 \text{ and }x_j s_j \ge 0\ {\rm  for}\ {\rm all}\ 1\le j\le d\}\quad\mbox{and also},
}\\
&{\partial\mathbb D_{d-2}:=\{x\in \partial \mathbb{D} : x_j= 0 \text{ for some }1 \leq j \leq d\}.}
\end{aligned}
\end{equation}
Notice that the subsets $\partial\mathbb D(s)$ for $s \in \{\pm1\}^d$ correspond to the different {\it faces} of the boundary $\partial \mathbb{D}$.

\subsection{Equality of $I_a$ and $I_q$ for small disorder.}

Here is our first main result. 
  
\begin{theorem}\label{theo:2} 
For any $d \geq 4$, $\kappa > 0$ and compact set $\mathcal K \subseteq \partial \mathbb{D} \setminus \partial \mathbb{D}_{d-2}$ 
there exists $\eps=\eps(d,\kappa,\mathcal K) > 0$ such that, for any RWRE satisfying Assumption $\mathrm{A}$ with ellipticity constant $\kappa$, if
 \begin{equation}\label{eq:theo:2}
\mathrm{dis}(\P)< \eps
\end{equation} then we have the equality $I_q(x)=I_a(x)$ for all $x \in \mathcal{K}$.
\end{theorem}

\begin{remark} \label{rem:res0}
	\textup{One can think of Theorem \ref{theo:2} above as saying that part of the region of equality in the boundary $\{ x \in \partial\mathbb D : I_q(x)=I_a(x)\}$ covers the whole of $\partial \mathbb{D} \setminus \partial \mathbb{D}_{d-2}$ in the limit as $\mathrm{dis}(\P) \rightarrow 0$ uniformly over all environmental laws $\P$ with a uniform ellipticity constant bounded from below by some $\kappa >0$. However, we remark that, for a {\it fixed} environmental law $\P$, $I_a$ and $I_q$ can never be equal everywhere in $\partial\mathbb{D}$ unless $\P$ is degenerate (i.e. $\omega$ is non-random under $\P$), see \cite[Proposition 4]{Y11}.}\qed
\end{remark}

Our next result states that there exists \textit{at least one} open neighborhood on which there is equality, whenever the environment satisfies the weaker condition of small enough \textit{imbalance}. More precisely, given $s \in \{\pm 1\}^d$ we define the \textit{imbalance of $\bbP$} on the face $\partial \mathbb{D}(s)$ as 
$$
\begin{aligned}
&\mathrm{imb}_s(\bbP):= \inf \bigg\{ \varepsilon > 0 : \zeta_s(x)\in [1-\eps,1+\eps], \,\,\bbP\text{-a.s. for all }x \in \Z^d\bigg\} \\
&\qquad\qquad\qquad\qquad\mbox{with     } \zeta_s(x):=\dfrac{\sum_{i=1}^d \omega(x,s_i e_i)}{\sum_{i=1}^d \alpha(s_i e_i)},
\end{aligned}
$$
or, equivalently, $\mathrm{imb}_s(\bbP)$ is the $L^\infty(\bbP)$-norm of the random variable $\zeta_s(x) -1$, for any given $x \in \Z^d$. 
Here is the statement of our next main result. 

\begin{theorem}\label{theo:1} For any $d \geq 4$, $\kappa >0$ and $s \in \{\pm 1\}^d$, there exists $\eps^\star=\eps^\star(d,\kappa)>0$ such that, for any RWRE satisfying Assumption $\mathrm{A}$ with ellipticity constant $\kappa$, if  
		\begin{equation}\label{eq:theo:1}
		\mathrm{imb}_s(\bbP) < \varepsilon^\star,
		\end{equation} then the following statements hold:
\begin{itemize}
		\item [$\bullet$] $I_a$ and $I_q$ have the same minimum over $\partial \mathbb{D}(s)$, 
	\begin{equation}\label{equality-infimum}
	\min_{x\in\partial\mathbb D(s)} I_q(x)=\min_{x\in \partial\mathbb D(s)} I_a(x) =-\log \sum_{i=1}^d \alpha(s_i e_i).
	\end{equation}
	\item [$\bullet$] $I_a$ and $I_q$ have the same unique minimizer, 
\[
{\arg\min}_{x\in\partial\mathbb D(s)} I_q(x) = {\arg \min}_{x\in\partial\mathbb D(s)} I_a(x) \\
=\frac{\sum_{i=1}^d \alpha(s_ie_i)s_ie_i}{\sum_{i=1}^d \alpha(s_i e_i)}=:\overline{x}_{s}.
\]
 \item [$\bullet$] There exists a neighborhood $\mathcal O\subset \partial\mathbb D(s)$ of $\overline{x}_s$ such that $I_a$ and $I_q$ agree on $\mathcal{O}$, 
\begin{equation} \label{eq:equality1}
I_q(x)=I_a(x) \qquad \text{for all }x\in \mathcal O.
\end{equation} Moreover, the set $\mathcal{O}$ can be taken to be uniform over all environmental laws $\bbP$ satisfying Assumption $\mathrm{A}$ with ellipticity constant $\kappa$ in the following sense: there exists $r=r(d,\kappa) > 0$ such that, for any $\bbP$ satisfying Assumption $\mathrm{A}$ with ellipticity constant $\kappa$, if $\mathrm{imb}_s(\bbP) < \eps^\star$ (with $\eps^\star$ as above) then 
\[
I_q(x) = I_a(x) \qquad\text{ for all }x \in B_r(\overline{x}_s) \cap \partial \mathbb{D}(s).
\] (The point being that $r$ is independent of $\overline{x}_s$ and uniform over  $\bbP$.)

\end{itemize}
\end{theorem}

\begin{remark} \label{rem:res}
\textup{Note that in the current general setup, we do not require the RWRE to possess any limiting velocity, nor do we impose any ballisticity condition on the RWRE. 
%Note that Theorem \ref{theo:1} contains as a particular case $\mathrm{imb}_s(\bbP)=0$ (certainly it also covers the case when $\mathrm{imb}_s(\bbP)$ is sufficiently small, not necessarily zero). 
However, one can show that, whenever \eqref{eq:theo:1}~holds, the unique minimizer $\overline{x}_s$ in Theorem \ref{theo:1} is the velocity of $(X_n)_{n \in \N_0}$ under the annealed conditional measure $P_{0}(\tfrac{X_n}{n} \in \cdot \,|\, \tfrac{X_n}{n} \in \partial \mathbb{D}(s))$ and the quenched and annealed rate~functions of the walk under this conditioning can be seen to equal $I_q - I_q(\overline{x}_s)$ and $I_a - I_a(\overline{x}_s)$, respectively. Also, under this conditioning, the set-up bears some resemblance 
to a random walk in a {\it space time i.i.d. environment} (\cite{Y09}) which corresponds to the case when $\mathrm{imb}_s(\bbP)=0$. The latter choice is  
included as a particular case of Theorem \ref{theo:1} (certainly Theorem \ref{theo:1} also covers the case when $\mathrm{imb}_s(\bbP)$ is sufficiently small, not necessarily zero). 
Also, from this viewpoint, our Theorem \ref{theo:2} then indicates that previously known equality results for dynamic random environments (available for neighborhoods of the velocity) can be extended to neighborhoods of \textit{arbitrary points} in the domain, provided that the disorder of the environment is sufficiently low (a fact which can be proved rigorously by an adaptation of our method).}\qed
\end{remark}

\begin{remark}\textup{Notice that if $\bbP$ is the law of a \textit{balanced random environment}, i.e. $\bbP$ is such that $\P(\omega(x,e)=\omega(x,-e) \text{ for all }x,e)=1$, then $\mathrm{imb}_s(\bbP)=0$ for any $s \in \{\pm1\}^d$. In particular, such environments, as well as small perturbations of them, readily satisfy the hypotheses of Theorem~\ref{theo:1}. Observe also that balanced random environments never satisfy condition-$(\mathrm{T})$ and, as such, had not been considered before in the study of equality of the rate functions for standard RWRE.}\qed
\end{remark}

\subsection{Formulas for $I_q$ and $I_a$ on the boundary.}
Using the observation that
the rate functions on the boundary $\partial\mathbb D$ can be studied as  
that of a random process in a space-time i.i.d. environment, Theorem \ref{theo:1} and Theorem \ref{theo:2}
now provide a simple formula for the quenched rate function $I_q$. 
Define the moment generating function $\lambda:\mathbb R^{d}\to\mathbb R$ as
\begin{equation}\label{lambda_theta}
\lambda(\theta):=\sum_{e\in \mathbb V}\alpha(e)\mathrm \e^{\langle \theta, e\rangle}.
\end{equation}

Here is our next main result.

\begin{theorem}\label{theo:4} Fix $d \geq 4$ and $\kappa > 0$. Then:
\begin{itemize}
\item[(i)] Given any compact set $\mathcal K \subseteq \partial \mathbb{D} \setminus \partial \mathbb{D}_{d-2}$ 
there exists $\eps=\eps(d,\kappa,\mathcal K) > 0$ such that, for any RWRE satisfying Assumption $\mathrm{A}$ with ellipticity constant $\kappa$, whenever \eqref{eq:theo:2} holds we have 
\begin{equation}\label{eq:theo:3b}
I_a(x)=I_q(x)=\sup_{\theta\in\mathbb R^{d}}\left(\langle \theta,x\rangle -\log\lambda(\theta)\right) =\sum_{i=1}^d |x_i| \log \frac{|x_i|}{\alpha(s_ie_i)}\qquad \text{ for all }x\in \mathcal K.
\end{equation}

\item[(ii)] Given any $s \in \{\pm 1\}^d$ there exists $\eps^\star=\eps^\star(d,\kappa)>0$ such that, for any RWRE satisfying Assumption $\mathrm{A}$ with ellipticity constant $\kappa$, whenever \eqref{eq:theo:1} holds there exists a non-empty open subset $\mathcal O \subset \partial\mathbb D(s)
\setminus\partial\mathbb D_{d-2}$ such that the representation in \eqref{eq:theo:3b} holds for all $x \in \mathcal{O}$. This open subset is the same from Theorem \ref{theo:1} and hence can be taken to be uniform over all $\bbP$ satisfying Assumption $\mathrm{A}$ with ellipticity constant $\kappa$.
\end{itemize}
\end{theorem} 

\begin{remark} \label{rem:eq}\textup{As a matter of fact, the formula 
	$$
	I_a(x)=\sup_{\theta \in \R^d}(\langle \theta, x\rangle - \log \lambda (\theta)) = \sum_{i=1}^d |x_i|\log \frac{|x_i|}{\alpha(\frac{x_i}{|x_i|}e_i)}
	$$ (with the convention that $0 \log 0=0$, used whenever $|x_i|=0$) in \eqref{eq:theo:3b} above holds for \textit{all} $x \in \partial \mathbb{D}$, not just for $x$ belonging to $\mathcal{K}$ or $\mathcal{O}$ (it is the equality with $I_q$ which only holds in $\mathcal{K}$ or $\mathcal{O}$, respectively). This will be evident from the proof of Theorem \ref{theo:4}.}\qed
\end{remark}  

\begin{remark}
\textup{The annealed rate function $I_a$ was shown in \cite{V03} to admit a variational formula involving entropy, which was analyzed further in \cite{PZ09,Y10,B12} under the additional assumption of condition-$(\mathrm{T})$. On the other hand, the quenched LDP in \cite{V03} was derived using sub-additivity methods which did not lead to any formula for $I_q$ (see also \cite{Z98} for the quenched LDP in the case of nestling environments in $d\geq 1$ and \cite{GdH98,CGZ00} for the $d=1$ case). Later, based on the method in \cite{KRV06}, the following variational formula for $I_q$ was shown in \cite{R06} for elliptic RWRE:}
\begin{equation}\label{Rosenbluth}
\begin{aligned}
&I_q(x)=F^\star(x)\stackrel{\mathrm{def}}=\sup_{\theta\in\mathbb R^d}[\langle \theta,x\rangle- F(\theta)]\quad\mbox{where},\\
&F(\theta)= \inf_G \,\mathrm{ess\,sup}_\P\, \log\bigg(\sum_{|e|=1} \omega(0,e) \mathrm e^{G(\omega,e)+\langle \theta,e\rangle}\bigg),
\end{aligned}
\end{equation}
\textup{where the infimum above is taken over a class of mean-zero gradients satisfying a certain moment condition. We also refer to \cite{Y08,RS11} for extensions of the above result to level-2 and level-3 LDP for elliptic RWRE,  and to \cite{BMO16} for a similar representation for non-elliptic RWRE including random walks on percolation clusters.  Finally, we refer to \cite{RSY17a,RSY17b} for another variational representation of the quenched rate function. Notice that the Cram\'er-type representation in \eqref{eq:theo:3b} simplifies its earlier antecedents significantly.}\qed
\end{remark}

\subsection{Monotonicity in the disorder and phase transition in the equality of rate functions.}\label{sec:monotonicity}

We now turn to the statement that provides a phase transition in the behavior of the difference $I_a(x,\cdot)- I_q(x,\cdot)$ as a function of the underlying disorder. We first need some further notation. Given a probability vector $\alpha \in \mathcal{M}_1(\bbV)$ with strictly positive entries, let 
$$
\mathcal E_\alpha:=\bigg\{(r(e))_{e \in \bbV} \in [-1,1]^{\mathbb V} :  
\sum_{e\in \mathbb V}\alpha(e)r(e)=0\ {\rm and}\ \sup_{e\in \mathbb V}|r(e)|= 1\bigg\}.
$$ 
We denote probability measures on the space $\Gamma_\alpha:=\mathcal E_\alpha^{\mathbb Z^d}$ by $\mathbb Q$.
We also write $\eta=(\eta(x))_{x\in\mathbb Z^d}\in \Gamma_\alpha$, with
$\eta(x)=(\eta(x,e))_{e\in \mathbb V}$ being a typical element of the space $\mathcal{E}_\alpha$. Since $\alpha$ will remain fixed in the remainder of this subsection, we will omit the dependence on $\alpha$ of $\Q$ and $\eta$ from the notation.

Now, given a probability vector $\alpha \in \mathcal{M}_1(\bbV)$ with strictly positive entries and a probability measure $\Q$ on $\Gamma_\alpha$, let us consider the parametrized family of random environments $\{\omega_\eps\}_{\eps\in [0,1)}$ given by
\[
\omega_\eps(x,e):=\alpha(e)(1+\eps\eta(x,e)).
\] We will make the following assumptions on $\Q$:

\medskip
\noindent{{\bf Assumption $\mathrm{B}$.}} The probability measure $\Q$ satisfies the following three properties:
\begin{enumerate}
	\item [$\bullet$] The support of $\Q$ is not a singleton.\footnote{As matter of fact, this condition is already implied by the third one since $\sup_{e \in \bbV} |\xi(x,e)|=1$ by definition of $\mathcal{E}_\alpha$. Nevertheless, we still include it for clarity purposes.}
	\item [$\bullet$] The family $(\eta(x))_{x\in\mathbb Z^d}$ is i.i.d. under $\Q$.
	\item [$\bullet$] $\E\eta(x,e)=0$ for all $e \in \mathbb{V}$ and $x \in \Z^d$.
\end{enumerate}

\medskip

The assumption that the support of $\Q$ is not a singleton is made to ensure that there exists some true randomness in the environments $\omega_\eps$ for $\eps>0$. On the other hand, the other two assumptions guarantee that for each $\eps \in [0,1)$ the law $\P_\eps$ of the environment $\omega_{\eps}$ satisfies Assumption $\mathrm{A}$ with ellipticity constant $\kappa:=(1-\eps)(\min_{e \in \bbV} \alpha(e)) >0$ and $\mathrm{dis}(\bbP_\varepsilon)=\varepsilon$, with $\E(\omega_\epsilon(x,e))=\alpha(e)$ for all $e \in \mathbb{V}$ and $x \in \Z^d$. In this context, we will denote by $I_a(\cdot,\eps)$ and $I_q(\cdot,\eps)$ to be the annealed and quenched rate functions, respectively. Recall that $I_a(x,\eps) \leq I_q(x,\eps)$ for all $x\in \Z^d$ and $\eps\geq 0$ by Jensen's inequality. Our next main result establishes the monotonicity property for the difference of these two rate functions $I_a(x,\cdot)- I_q(x,\cdot)$.

\begin{theorem}\label{theo:5} {Fix $d \geq 4$. Then, for any probability vector $\alpha \in \mathcal{M}_1(\mathbb{V})$ with strictly positive entries and probability measure $\Q$ on $\Gamma_\alpha$ satisfying Assumption $\mathrm{B}$, the following assertions hold}:
\begin{itemize}
\item {For each $x\in\partial \mathbb D $, 
the map 
$$
[0,1)\ni \eps \mapsto I_a(x,\eps)- I_q(x,\eps)
$$ is non-increasing and continuous.}
In particular, there is $\eps_c(x)\geq 0$ such that 
 for $\eps\in[0,1)$,
\begin{equation}\label{phasetransition}
 \begin{cases}
I_{a}(x,\eps)=I_{q}(x,\eps) \quad \mbox{if}\quad \eps\leq \eps_{c}(x)\\
I_{a}(x,\eps)<I_{q}(x,\eps) \quad\mbox{if} \quad \eps>\eps_{c}(x). 
\end{cases}
\end{equation} 
\item Furthermore, there exists an open subset $\mathcal O \subset \partial \mathbb D  \setminus \partial \mathbb{D}_{d-2}$
such that for all $x\in \mathcal O$,
\begin{equation}\label{truephasetransition}
0<\eps_c(x)<1.
\end{equation} 
\end{itemize}
\end{theorem} 

\begin{remark} \textup{It follows from Theorem \ref{theo:1} that for any $x \in \partial \mathbb D  \setminus \partial \mathbb{D}_{d-2}$ one always has $\eps_c(x) > 0$. What is (in principle) only true for $x \in \mathcal{O}$ is the additional requirement in \eqref{truephasetransition} that $\varepsilon_c(x) < 1$, which together with $\varepsilon_c(x)>0$ implies the existence of a true phase transition in the disorder $\varepsilon$.}\qed
\end{remark}

\begin{remark}
\textup{We emphasize that the family of random environments considered presently is quite general and contains several widely studied models for RWRE (see \cite{CR17,S04}). Furthermore, consideration of such a parametrization is in fact quite natural. Indeed, there are two basic questions that one can ask regarding this point. Namely, 
\begin{enumerate}
\item Given $x \in (\partial \mathbb D\setminus \partial\mathbb D_{d-2})$, is it true that there exists $\eps_x$ such that the equality $I_a(x)=I_q(x)$ holds for {\it any} model with disorder less than $\eps_x$ and fails to hold for all larger disorders?
\item Given $x \in (\partial \mathbb D\setminus \partial\mathbb D_{d-2})$, is the mapping $\eps\mapsto I_a(x,\eps)-I_q(x,\eps)$ monotonic?
\end{enumerate}
Clearly the affirmation of (2) implies the same for (1). However, (2) does not make sense in general. Indeed, $I_a$ and $I_q$ need not be functions of the underlying disorder, only perhaps when dealing with {\it parametrized families of environments} as in Theorem \ref{theo:5}. On the other hand, (1) does make sense in general, but it seems out of reach with our current method and we are not sure even if it is true. The difference with our Theorem \ref{theo:5} is that for us the ``source of randomness" is fixed beforehand, so that when we make its influence smaller and smaller by taking the limit $\varepsilon_x \to 0$ then it is natural to expect equality to hold. However, we do not know whether there exists some {\it universal} $\eps_x$ which works simultaneously for all possible sources of randomness (as the affirmation of (1) would imply).}

\end{remark}

\subsection{Outline of the proofs}\label{sec:proof:sketch}
For the sake of conceptual transparency and also to provide guidance to the reader, we find it convenient to present a brief description of the method of~proof developed in the present article. This will then also underline the technical novelty of our contribution. 

To treat the boundary behavior of $I_q$ and $I_a$, we shall develop a somewhat different approach to the one used in \cite{BMRS19} to deal with the behavior in the interior of $\mathbb D$. The method in the interior used there relied on the construction of an auxiliary random walk in a deterministic environment possessing a regeneration structure and showing that its large deviation properties are intimately related to those of the true RWRE. Since the RWRE behaves differently on the boundary,\footnote{While it might be possible to again define an auxiliary walk and study its regeneration times on the boundary, many technical problems now appear due to the non-positive definiteness of the Hessian of (the averaged) logarithmic moment generating function as the support of the first step for the auxiliary walk on the boundary is contained in a $(d-1)$-dimensional hyperplane, in addition to the reduced dimension $d-1$ leading to additional difficulties in using the approach of \cite{BMRS19} which requires that the dimension be at least four.} here we develop an alternative approach which is conceptually more transparent and is based on a novel application of the {\it martingale method} developed originally by Bolthausen \cite{B89} in the context of directed polymers \cite{C17}. The key idea is to construct the ``renormalized partition function" or the {\it polymer martingale} in the context of general RWRE scenario even in the absence of ``directed" structure. To this end, first we observe that it is enough to show equality of the rate functions holds on each face separately, i.e. for compact sets $\mathcal{K} \subseteq \partial \mathbb D \setminus \partial \mathbb D_{d-2}$ contained in $\partial\mathbb D(s)$ for some $s=( s_1,\dots, s_d)\in \{\pm 1\}^d$, where
$$
\partial\mathbb D(s):=\{x\in \partial\mathbb D\colon s_j x_j\geq 0\,\,\forall j=1,\dots,d\}.
$$ At this point, we make the following crucial observation: for each $s \in \{\pm\}^d$, on the event
\[
\mathbb B_n(s):= \big\{\tfrac{1}{n}(X_n-X_0) \in \partial \mathbb D(s)\big\},
\] one has that for all $j=1,\dots,n$
\begin{equation}
\label{obs1}
X_j - X_{j-1} \in \mathbb{V}(s):=\{s_i e_i : i=1,\dots,d\}
\end{equation} and, as a consequence, that for any $j,j' \in \{0,\dots,n\}$
\begin{equation} \label{obs2}
X_j=X_{j'} \Longleftrightarrow j=j'.
\end{equation}
In particular, if for an affine transformation $\pi$ mapping the hyperplane $\{x\colon \sum_{j=1}^d s_j x_j =1\}$ which contains $\partial \mathbb{D}(s)$ onto $\{x\colon x_d=0\}$ we define 
the {\it projected RWRE} $S_n:=\sum_{j=0}^{n-1} \pi(X_{j+1}-X_j)$ then, on the event $\mathbb{B}_n(s)$, the walk $S_n$ satisfies the following two important properties:
\begin{itemize}
	\item By \eqref{obs1}, the path $(S_1,\dots,S_n)$ falls entirely on the hyperplane $\{x\colon x_d=0\}=\R^{d-1}\times \{0\}$, and therefore we may view it as a $(d-1)$-dimensional walk. Moreover, since the jumps $(\pi(e))_{e \in \mathbb{V}(s)}$ of $S_n$ span all of $\{x : x_d = 0\}$, it has effective dimension $d-1$.
	\item For each $j=1,\dots,n$, the weights used by $S_j$ to decide where to jump next are given by the random probability vector $\omega(X_{j-1},X_{j}-X_{j-1})$. By the i.i.d. structure of the environment, \eqref{obs2} yields that these vectors $(\omega(X_{j-1},X_{j}-X_{j-1}))_{j=1,\dots,n}$ are independent. Furthermore, by uniform ellipticity, all these weights are uniformly bounded away from $0$.
\end{itemize}

These crucial facts now allow us to construct a non-negative martingale {\it on the event $\mathbb B_n(s)$} which in our context translates to 
$$
\mathscr Z_{n,\theta}(\omega,x):= \psi^{-n}(\theta) E_{x,\omega}\big[\mathrm \e^{\langle \theta, S_n\rangle}\,\mathbbm 1_{\mathbb B_n(s)}],\qquad\mbox{with}\quad \psi(\theta):= \sum_{e\in \mathbb{V}(s)} \alpha(e) \mathrm \e^{\langle \theta, \pi(e)\rangle}.
$$
The above structure seems to be a natural way to construct the ``renormalized partition function" in the context of general RWRE. 
However since the above extra ubiquitous conditions (e.g. restriction to paths on $\mathbb B_n(s)$) 
manifest throughout the entire analysis,  the actual leveraging of the martingale method in our context of Theorem \ref{theo:2} (cf. Section \ref{sec:proof:theo:2} for its proof) 
and Theorem \ref{theo:1} (cf. Section \ref{sec:proof:theo:1} for its proof) is quite 
different from earlier approaches. Theorem \ref{theo:4} then follows from the proof of the two earlier results, while the proof of Theorem \ref{theo:5} builds on a method relying on the FKG inequality, see Section \ref{sec:proof:theo:5} for the proofs of these two results.

\section{Equality on the boundary $\partial\mathbb D$ - Proof of Theorem \ref{theo:2}}\label{sec:proof:theo:2}

We first remark that the boundary $\partial \mathbb{D}$ of the unit ball $\mathbb D$ can be decomposed into (non-overlapping) \textit{faces} $\partial \mathbb{D}(s)$, $ s=( s_1,\dots, s_d) \in \{-1,1\}^d$, defined as
$$
\partial \mathbb{D}(s):=\{ x \in \partial \mathbb{D} : s_j x_j \geq 0 \text{ for all }j=1,\dots,d\}.
$$ 
We will prove the equality of rate functions
\begin{equation} \label{eq:equality}
I_q(x)=I_a(x)
\end{equation} 
under the assumptions of Theorem \ref{theo:2} on each face $\partial \mathbb{D}(s)$ separately. Since the proof is exactly the same for all faces, from now on we will fix a face ${s}:=(s_1,\dots,s_d)$ and prove \eqref{eq:equality} for $x \in \partial \mathbb{D}({s})$. For simplicity, in the sequel we will also sometimes remove the dependence on ${s}$ from the notation. 

Our proof of \eqref{eq:equality} is divided into four steps, each occupying a separate subsection. Before we begin, let us introduce some further notation to be used throughout the sequel. Given $\kappa > 0$, we define 
\[
\mathcal{M}_1^{(\kappa)}(\bbV):= \{ p \in \mathcal{M}_1(\bbV) : p(e)\geq \kappa \text{ for all }e \in \bbV\},
\] together with the class of environmental laws 
\[
\mathcal{P}_\kappa:=\{ \bbP \in \mathcal{M}_1(\Omega) : \bbP \text{ satisfies Assumption $\mathrm{A}$ with ellipticity constant $\kappa$}\},
\] where $\mathcal{M}_1(\Omega)$ is the space of all environmental laws. We are now ready to begin the proof.

\subsection{Projecting on a $(d-1)$-dimensional hyperplane.}\label{sec:proj}

For each $n \in \N_0$ let us define 
\begin{equation}
\label{eq:borde}
\partial R_n:=\{x \in \Z^d : |x|=n\,,\,s_j x_j \geq 0 \text{ for all }j=1,\dots,d\}= n \cdot \partial \mathbb{D}({s}).
\end{equation} and for each $x \in \Z^d$ set 
\[
\partial R_n(x):= x + \partial R_n.
\] Also, define the set $\mathbb{V}({s})$ of ${s}$-allowed jumps as 
\[
\mathbb{V}({s})=\{ s_j e_j : j=1,\dots,d\} \subseteq \mathbb{V}.
\] 
Given $n \geq 1$, recall that a sequence $z:=(z_0,\dots,z_n)$ of sites in $\mathbb Z^d$ is a {\it path of length $n$} if $z_j-z_{j-1} \in \bbV$ for all $j=1,\dots,n$. For $x\in \Z^d$, let $\mathcal{R}_n(x)$ denote the set of all paths of length $n$ such that $z_0=x$ and $z_n \in \partial R_n(x)$. Notice that a path $z=(z_0,\dots,z_n)$ of length $n$ belongs to $\partial R_n(z_0)$ if and only if all of its jumps belong to $\bbV({s})$, i.e. if we define the $j$-th jump of the path $z$ by 
\begin{equation}\label{ez}
\Delta_j(z):=z_{j}-z_{j-1},
\end{equation} then 
\begin{equation}\label{eq:equivalence0}
z=(z_0,\dots,z_n) \in \mathcal{R}_n(z_0) \Longleftrightarrow \Delta_j(z) \in \mathbb{V}({s}) \text{ for all }j=1,\dots,n,
\end{equation} from where we easily deduce that
\begin{equation} \label{eq:obs}
z=(z_0,\dots,z_n) \in \mathcal{R}_n(z_0) \Longleftrightarrow (z_0,\dots,z_{n-1}) \in \mathcal{R}_{n-1}(z_0) \text{ and }\Delta_{n}(z) \in \mathbb{V}(s).
\end{equation} 

Now, notice that $\{x : {s}_1x_1+\dots+{s}_dx_d = 1\}$ is the unique hyperplane which contains $\mathbb{V}({s})$, which is (affinely) generated by the vectors $({s}_i e_i)_{i=1,\dots,d}$, and let $\pi : \R^d \to \R^d$ be the affine transformation mapping $\{x : {s}_1x_1+\dots+{s}_dx_d = 1\} \longrightarrow \{x : x_d = 0\}$ given by 
\begin{equation} \label{eq:pi}
\begin{aligned}
\pi(x)=\begin{cases}
e_i & \text{ if }x={s}_ie_i \text{ for }i=1,\dots,d-1
\\
-(e_1+\dots+e_{d-1}) & \text{ if }x={s}_d e_d\\
\tfrac{d-1}{d}e_d & \text{ if }x={s}. \end{cases}
\end{aligned}
\end{equation} We then define then the \textit{projected walk} $(S_n)_{n \in \N}$ by the formula
\begin{equation}\label{Sn}
S_n := \sum_{j=1}^n \pi(X_j-X_{j-1}), \hspace{1cm}k \in \N_0,
\end{equation}
 where $X=(X_n)_{n \in \N_0}$ is our original RWRE, and for each $n \geq 1$ consider the event 
\begin{equation}\label{eventBn}
\mathbb B_n:=\{ \Delta_j(X) \in \mathbb{V}(s) \text{ for all }j=1,\dots,n\} = \{ (X_0,\dots,X_n) \in \partial R_n(X_0)\}.
\end{equation}
Notice that, on the event $\mathbb B_n$, the projected walk $S_n$ belongs to the hyperplane $\{x \in \R^d : x_d = 0\}$, which we can (and will henceforth) identify with $\R^{d-1}$.%\footnote{Note that, under the identification $\{x \in \R^d : x_d = 0\}=\R^{d-1}$, on the event $\mathbb{B}_n$ the projected walk~$S_n$ lives on the oriented graph $\mathcal{G}$ with vertex set $\Z^{d-1}$ and whose set of oriented edges~$E$ is given by the relation $(x \to y) \in E \Longleftrightarrow y=x+ \pi(s_ie_i)$ for some $i=1,\dots,d$. On the other hand, it follows from the proof of Lemma \ref{lema:1} that, given any random walk $(Y_n)_{n \in \N}$ in a~space-time i.i.d. environment $\overline{\omega}=(\overline{\omega}_n)_{n \geq 0}$ on $\mathcal{G}$, there always exists some random walk $(X_n)_{n \in \N}$ in a \textit{balanced} i.i.d. random environment $\omega$ on $\Z^d$ such that its projected walk $S_n$ conditioned on $\mathbb{B}_n$ has the same law as $Y_n$. Indeed, it suffices to define $\omega$ through the formula $\omega(0,-s_ie_i)=\omega(0,s_ie_i)\overset{d}{=}\tfrac{1}{2}\overline{\omega}_n(x \to x+\pi(s_ie_i))$ for any $x \in \mathcal{G}$ and $n \geq 0$. This tells us that random walks in dynamical random environments on the $(d-1)$-dimensional graph $\mathcal{G}$ are a particular case in the current setup under investigation.}
Thus, if for $\theta \in \R^{d-1}$ we define
\begin{equation}\label{psi} 
\psi(\theta):= \sum_{e \in \mathbb{V}({s})} \alpha(e)\e^{\langle \theta, \pi(e)\rangle} = \sum_{i=1}^{d-1} \alpha(s_i e_i)\e^{\theta_i} + \alpha({s}_d e_d)\e^{-(\theta_1+\dots+\theta_{d-1})},
\end{equation} with the identification $\{x \in \R^d : x_d = 0\}=\R^{d-1}$ in mind we may define for $n \in \N$ and $x \in \Z^d$, 
\begin{equation}\label{scrZ}
\begin{aligned}
\mathscr Z_{n,\theta}(\omega,x)&:=\frac{ E_{x,\omega}(\e^{\langle \theta, S_n\rangle} \mathbbm 1_{\mathbb B_n})}{\psi^n(\theta)}\\
&=\frac{\sum_{z\in\mathcal R_n (x)} \e^{\langle\theta,\sum_{j=1}^{n} \pi(\Delta_j(z))\rangle} \prod_{i=1}^n \omega(z_{j-1},\Delta_j(z))}{\psi^{n}(\theta)}.
\end{aligned}
\end{equation}
for $\Delta_j(z)$ as in \eqref{ez}. Now a simple computation using \eqref{eq:obs} and the definition of $\psi$ shows that 
$$
\mathscr Z_\theta(\cdot)=(\mathscr Z_{n,\theta}(\cdot,x))_{n \in \N}
$$
 is a $\P$-martingale for any $\theta$ and $x$. Being also nonnegative, we know it has an $\P$-almost sure limit:
 \begin{equation}\label{scrZinfty}
 \mathscr Z_{\infty,\theta}(\cdot,x) \stackrel{\mathrm{a.s.}}= \lim_{n\to\infty} \mathscr Z_{n,\theta}(\cdot,x).
 \end{equation}

\subsection{Martingale convergence in $L^2$.}
Our goal is now to show that the converge in \eqref{scrZinfty} holds also in $L^2(\P)$. The following assertion, providing the desired $L^2(\P)$-convergence, will furthermore imply that the limit $\mathscr Z_{\infty,\theta}$ is also strictly positive. 

Recall the definition of disorder $\mathrm{dis}(\bbP)$ from \eqref{eq:defdis}. 
\begin{lemma}\label{lema:1} 
Given $d \geq 4$, $\kappa > 0$ and a compact set $\Theta \subseteq \R^{d-1}$, there exists $\eps^\prime=\eps'(d,\kappa,\Theta)>0$ such that, for any RWRE in dimension $d$ with $\bbP \in \mathcal{P}_\kappa$, if $\mathrm{dis}(\P) < \eps^\prime$ then for any $x \in \Z^d$ 
\[
\sup_{n \in \N\,,\,\theta \in \Theta}\|\mathscr Z_{n,\theta}(x)\|_{L^2(\mathbb P)}<\infty.
\]
\end{lemma}

For the proof of Lemma \ref{lema:1} we shall need the following result, which is (a particular version of) the well-known Khas'minskii's lemma. We include the short proof to
keep the material self-contained. 

\begin{lemma}\label{lema:2} Let $Z=(Z_i)_{i \in \N_0}$ be a random walk on $\Z^d$ starting at the origin, whose law is denoted by $\mathrm P_0$ with expectation $\mathrm E_0$.   If we define
	$$
	\eta:=  \mathrm E_0\left(\sum_{i=0}^\infty \mathbbm{1}_{\{Z_i=0\}}\right) = \sum_{i=0}^\infty P_0(Z_i=0)
	$$ then for any $C > 0$ such that $C\eta < 1$ we have 
	\begin{equation}\label{eq:exp}
	\mathrm E_0\left( \exp\left\{ C \sum_{i=0}^\infty \mathbbm{1}_{\{Z_i=0\}}\right\}\right) \leq \frac{1}{1-C\eta}.
	\end{equation}	
\end{lemma}

\begin{proof}
	By expanding the exponential on the left-hand side in \eqref{eq:exp} we can write 
	$$
	\begin{aligned}
	\mathrm E_0\left( \exp\left\{ C \sum_{i=0}^\infty \mathbbm{1}_{\{Z_i=0\}}\right\}\right)& = \sum_{n=0}^\infty \frac{C^n}{n!} \mathrm E_0\left[ \left(\sum_{i=0}^\infty \mathbbm{1}_{\{Z_i=0\}}\right)^n\right]\\
	 & \leq \sum_{n=0}^\infty C^n \sum_{0\leq i_1\leq \dots \leq i_n} \mathrm P_0( Z_{i_1}=0,\dots,Z_{i_n}=0)\\
	 & = \sum_{n=0}^\infty C^n \sum_{0\leq i_1\leq \dots \leq i_{n-1}}  \mathrm P_0( Z_{i_1}=0,\dots,Z_{i_{n-1}}=0) \sum_{i_n= i_{n-1}}^\infty \mathrm P_0(Z_{i_n-i_{n-1}}=0)\\
	& = \sum_{n=0}^\infty C^n \eta \sum_{0\leq i_1\leq \dots \leq i_{n-1}}  \mathrm P_0( Z_{i_1}=0,\dots,Z_{i_{n-1}}=0) \\
	& = \sum_{n=0}^\infty (C \eta)^n = \frac{1}{1-C\eta} \qquad\mbox{if}\quad C\eta<1,
	\end{aligned}
	$$ where in the upper bound above we have used symmetry, while the next identities follow by successive use of the Markov property.
\end{proof}

We are now ready to prove Lemma \ref{lema:1}.

\begin{proof}[{\bf Proof of Lemma \ref{lema:1}}] By the translation invariance of the environment, it will suffice to show the claim for $x = 0$ and, for notational convenience, in the sequel we will abbreviate 
$\mathcal{R}_n:= \mathcal{R}_n(0)$ and $\mathscr Z_{n,\theta}:=\mathscr Z_{n,\theta}(0)$. Then
	\begin{equation}\label{eq:G}
	\begin{aligned}
	 \|\mathscr Z_{n,\theta}\|_{L^2(\mathbb{P})}^2
	&= \frac{\E(E^2_{0,\omega}(\mathrm e^{\langle \theta, S_n\rangle}\mathbbm{1}_{B_n}))}{\psi^{2n}(\theta)}
	\\
	&= \sum_{ z,z' \in \mathcal{R}_n} \E\left(\left(\prod_{j=1}^{n} \omega(z_{j-1},\Delta_j(z))\frac{\mathrm e^{\langle \theta,\pi(\mathrm \Delta_j(z))\rangle}}{\psi(\theta)}\right)\left(\prod_{k=1}^{n} \omega(z'_{k-1},\Delta_k(z'))\frac{\mathrm e^{\langle \theta,\pi(\Delta_k(z'))\rangle}}{\psi(\theta)}\right)\right)
	\end{aligned}
	\end{equation} 
	Now the following simple observation is crucial for our context. By \eqref{eq:equality} we have that 
	\begin{equation} \label{eq:cruobs}
	z=(0,\dots,z_n) \in \mathcal{R}_n \Longrightarrow |z_j|=j \text{ for all }j=1,\dots,n,
	\end{equation} so that the $z_j$ must be all distinct and, furthermore, for $z,z' \in \mathcal{R}_n$ one has $z_j=z'_k$ only if $j=k$. 

Using that our environment is i.i.d., this allows us to rewrite \eqref{eq:G} as 
	\begin{equation}\label{eq:G2}
	\begin{aligned}
	 \|\mathscr Z_{n,\theta}\|_{L^2(\mathbb{P})}^2
	&= \sum_{z,z' \in \mathcal{R}_n} \E\left( \prod_{j=1}^n \left( \omega(z_{j-1},\Delta_j(z))\omega(z'_{j-1},\Delta_j(z'))\frac{\mathrm e^{\langle \theta,\pi(\Delta_j(z))\rangle}}{\psi(\theta)}\frac{\mathrm e^{\langle \theta,\pi(\Delta_j(z'))\rangle}}{\psi(\theta)}\right)\right)\\
	&=\sum_{z,z' \in \mathcal{R}_n} \prod_{j=1}^n \left( \E\left(\omega(z_{j-1},\Delta_j(z))\omega(z'_{j-1},\Delta_j(z'))\right)\frac{\mathrm e^{\langle \theta,\pi(\Delta_j(z))\rangle}}{\psi(\theta)}\frac{\mathrm e^{\langle \theta,\pi(\Delta_j(z'))\rangle}}{\psi(\theta)}\right). 
	\end{aligned}
	\end{equation} Now, define the probability vector $\vec{\alpha}^{\ssup \theta}=(\alpha^{\ssup\theta}(\pi(e)))_{e \in \mathbb{V}(s)}$ on $\R^{d-1}$ by the formula
	\begin{equation}\label{qtheta}
	\alpha^{\ssup \theta}(\pi(e)):=\alpha(e)\frac{\mathrm e^{\langle \theta, \pi(e)\rangle}}{\psi(\theta)},
	\end{equation}
	 and $P^{\ssup\theta}_0$ as the law of the random walk on $\R^{d-1}$ starting from $0$ having jump distribution $\vec{\alpha}^{\ssup\theta}$. Then, since 
	$$
	\E\left(\omega(z_{j-1},\Delta_j(z))\omega(z'_{j-1},\Delta_j(z'))\right)=\alpha(\Delta_j(z))\alpha(\Delta_j(z'))
	$$ holds by independence whenever $z_{j-1}\neq z'_{j-1}$, a straightforward computation yields that one can rewrite \eqref{eq:G2} as
	$$
	 \|\mathscr Z_{n,\theta}\|_{L^2(\mathbb{P})}^2= E_0\left(\exp{\left\{\sum_{j=1}^n \mathbbm{1}_{\{X^{\ssup\theta}_{j-1}=Y^{\ssup\theta}_{j-1}\}}V(X^{\ssup\theta}_j-X^{\ssup\theta}_{j-1},Y^{\ssup\theta}_{j}-Y^{\ssup\theta}_{j-1})\right\}}\right)
	$$ where $X^{\ssup \theta}$ and $Y^{\ssup \theta}$ are two independent random walks with law $P^{\ssup\theta}_0$ and expectation $E^{\ssup\theta}_0$, and for $e,e' \in \mathbb{V}(s)$ we write 
	$$
	V(\pi(e),\pi(e')):=\log \left(\frac{\E(\omega(0,e)\omega(0,e'))}{\alpha(e)\alpha(e')}\right).
	$$ Note that $V$ is well-defined by uniform ellipticity and, moreover, since ${\omega(0,e)\leq \alpha(e)(1+\mathrm{dis}(\bbP))}$ for each $e \in V$, we have an upper bound
	$$
	V(\pi(e),\pi(e'))\leq {\log(1+\mathrm{dis}(\bbP)) \leq \mathrm{dis}(\bbP)},
	$$ implying that
	$$
	\|\mathscr Z_{n,\theta}\|_{L^2(\mathbb{P})}^2\leq E_0\left( \exp\left\{ \mathrm{dis}(\bbP) \sum_{j=0}^{n-1} \mathbbm{1}_{\{Z^{\ssup\theta}_j=0\}}\right\}\right)
	$$ where, for $j=0,\dots,n-1$, we write $Z^{\ssup\theta}_j=X^{\ssup\theta}_j-Y^{\ssup\theta}_j$. In particular, we see that
	\begin{equation}
	\label{eq:bound1b}
	\sup_{n \in \N \,,\, \theta \in \Theta} \|\mathscr Z_{n,\theta}\|_{L^2(\mathbb{P})}^2 \leq \sup_{\theta \in \Theta} E_0\left( \exp\left\{\mathrm{dis}(\bbP) \sum_{j=0}^\infty \mathbbm{1}_{\{Z^{\ssup\theta}_j=0\}}\right\}\right).
	\end{equation} By Lemma \ref{lema:2}, the right-hand side of \eqref{eq:bound1b} will be finite if 
	\[\sup_{\theta \in \Theta} \left( \sum_{j=0}^\infty P_0( Z^{\ssup\theta}_j=0)\right) < \frac{1}{\mathrm{dis}(\bbP)}.
	\]

Now, let $\chi_{\theta}(\xi)=E_0^{\ssup\theta}[\exp\{ \mathbf i \langle  \xi, Z_1^{\ssup\theta}\rangle \}]$ denote the characteristic function of $Z_1^{\ssup\theta}$ (recall that $Z_0^{\ssup\theta}=0$). Since  $Z_1^{\ssup\theta}=X_1^{\ssup\theta}-Y_1^{\ssup\theta}$ with $X_1^{\ssup\theta},Y_1^{\ssup\theta}$ i.i.d., $\chi_\theta$ takes only real non-negative values. We claim that there exists a $C_d> 0$ depending only on $d$ such that, for any $\theta \in \R^{d-1}$ and $r>0$, 
\begin{equation} \label{eq:cotap}
\sum_{j=0}^\infty P_0( Z^{\ssup\theta}_j=0) \leq C_d r^{-(d-1)} \int_{B_r} \frac{\d \xi}{1-\chi_\theta(\xi)},
\end{equation} where $B_r:=\{ \xi \in \R^{d-1} : |\xi| \leq r\}$. 

We defer the proof of \eqref{eq:cotap} and continue with the proof of Lemma \ref{lema:1}. 
Note that the support of $|Z_1^{\ssup\theta}|$ is uniformly bounded in $\theta$. Therefore, by Taylor's expansion we have 
\begin{equation}
\label{eq:cotaphi}
\chi_\theta(\xi) \leq 1 - \tfrac12 \sum_{i,k=1} a^\theta_{ik} \xi_i \xi_k + C |\xi|^3
\end{equation} for some constant $C>0$ independent of $\theta$, where $(a^{\ssup\theta}_{ik})_{i,k}$ is the covariance matrix of $Z^{\ssup\theta}_1$. Finally, since $(a^{\ssup\theta}_{ik})_{i,k}$ is positive definite for each $\theta$ {(since the random walk $Z^\theta$ has effective dimension $d-1$)} and the maps 
$$
(\alpha,\theta) \mapsto a^{\ssup \theta}_{ik}
$$
 are continuous for all $i,k$, by proceeding as in the proof of Lemma \ref{lemma:unifstar}, it follows from \eqref{eq:cotaphi} that for any compact set $\Theta \subset \R^{d-1}$ there exist $r_0=r_0(d,\kappa,\Theta), c_0=c_0(d,\kappa,\Theta)> 0$  such that 
$$
c_0 |\xi|^2 \leq 1-\chi_\theta(\xi)
$$ for all $\xi \in B_{r_0}$. In particular, from \eqref{eq:cotap} we see that, since $d \geq 4$, for some constant $\overline{C}_d>0$ depending only on $d$ we have
\begin{equation}
\label{eq:boundsumtheta}
\sup_{\theta \in \Theta} \sum_{j=0}^\infty P_0\big( Z^{\ssup\theta}_j=0\big) \leq C_d \frac{r_0^{-(d-1)}}{c_0} \int_{B_{r_0}} \frac{1}{|\xi|^2}\,\mathrm d\xi= \overline{C}_d \frac{r_0^{-(d-1)}}{c_0}\int_0^{r_0}{r^{d-4}}=:C_0 <\infty.
\end{equation} 
Taking $\eps':=\frac{1}{C_0}$ then yields the result. We now owe the reader only the proof of the claim \eqref{eq:cotap}. But this is an immediate consequence of Lemma \ref{lemma:Fourier} below, which is a well-known application of the Fourier inversion formula. 
\end{proof}

\begin{lemma}\label{lemma:Fourier}
	Let $(Z_n)_{n\geq 0}$ be a random walk in $\R^d$ with law $\mathrm P_0$ starting at the origin and assume that $\chi_{\mu}$, the characteristic function of $Z_1$, takes only real non-negative values. Then for any $r>0$ and $\delta= \sqrt d/r$, 
	$$
	\sum_{n\geq 0} \mathrm{P}_0\big[Z_n\in B_\delta(0)\big] \leq \frac {C_d}{r^d}  \int_{B_r(0)} \frac {\mathrm d \xi}{1-\chi_\mu(\xi)}.
	$$
\end{lemma}
\begin{proof}
	Since we are interested in the event $\{Z_n \leq \delta\}$ we need to consider the function $\prod_{j=1}^d f(x_j/\delta)$ where $f(x_j)=\max(1-|x_j|,0)$. Then we have the Fourier transform of 
	the product 
	$$
	\widehat{\prod_{j=1}^df(x_j)}=\prod_{j=1}^d \widehat f(\xi_j)\quad\mbox{with} \quad\widehat f(\xi_j)= \frac 2 {\xi_j^2} (1- \cos\xi_j).
	$$
	If $\mu$ denotes the law of $Z_1$ and $\mu^{\star n}=\mu\star\dots\star\mu$ its $n$-fold convolution, then for any $\delta>0$, \footnote{Recall that if $\mu$ and $\nu$ are two probability measures on $\R^d$ with charactersitic functions $\chi_\mu$ and $\chi_\nu$ respectively, then $\int \chi_\nu(x) \mu(\mathrm d x)= \int \chi_\mu(\xi) \nu(\mathrm d\xi)$.}
	$$
	\int_{\R^d} \widehat{\prod_{j=1}^d f\big(\frac{x_j}\delta\big)} \mu^{\star n}(\mathrm d x)= \delta^d \int \prod_{j=1}^d f(\delta \xi_j) (\chi_\mu(\xi))^n \, \mathrm d\xi.$$
	Therefore, for any $a\in (0,1)$, 
	\begin{equation}\label{Fourier1}
	\int_{\R^d} \widehat{\prod_{j=1}^d f\big(\frac{x_j}\delta\big)} \sum_{n\geq 0} a^n \mu^{\star n}(\mathrm d x)= \delta^d \int  \frac{\prod_{j=1}^d f(\delta \xi_j)}{1-a \chi_\mu(\xi)} \mathrm d \xi,
	\end{equation} which implies that, for $\delta=\sqrt d/r$ and a suitable constant $C>0$, 
	\[
	\begin{aligned}
	\sum_{n\geq 0} \mathrm{P}_0\big[Z_n\in B_\delta(0)\big]= \sum_{n\geq 0} \mu^{\star n}(B_\delta(0))  \leq C \int_{\R^d} \widehat{\prod_{j=1}^d f\big(\frac{x_j}\delta\big)} \sum_{n\geq 0}  \mu^{\star n}(\mathrm d x) 
	&= C \delta^d \sup_{a\in (0,1)} \int \frac{\prod_{j=1}^d f(\delta \xi_j)}{1-a \chi_\mu(\xi)} \mathrm d \xi.
	\\
	&\leq C_d r^{-d} \int_{B_r(0)} \frac {\mathrm d \xi}{1-\chi_\mu(\xi)}.
	\end{aligned}
	\]
\end{proof}

\subsection{Strict positivity of the limit $\mathscr Z_{\infty,\theta}$}

The next step in the proof is to show the martingale limit $\mathscr Z_{\infty,\theta}$ is strictly positive.

\begin{proposition}\label{prop:pos} Given $d \geq 4$, $\kappa >0$ and a compact set $\Theta \subseteq \R^{d-1}$ we have that, for any RWRE in dimension $d$ with $\bbP \in \mathcal{P}_\kappa$, if $\mathrm{dis}(\bbP) < \eps'$ (with $\eps'$ as in Lemma \ref{lema:1}) then for each $\theta \in \Theta$, 
	$$
	\P\big\{ \mathscr Z_{\infty,\theta}(x)> 0 \text{ for all }x \in \Z^d\big\}=1.
	$$
\end{proposition}

\begin{proof} By \eqref{eq:equivalence0} we have 
	$$
	z=(0,z_1,\dots,z_n) \in \mathcal{R}_n \Longleftrightarrow \Delta_1(z) \in \mathbb{V}(s) \text{ and }(z_1,\dots,z_n) \in \partial R_{n-1}(z_1)
	$$ so that, by conditioning on the first step of the walk $X_1$, a straightforward computation yields that
	\begin{equation} \label{eq:prelim1}
	\mathscr Z_{n,\theta}(\omega,0)=\sum_{e \in \mathbb{V}(s)} \omega(0,e)\mathrm e^{\langle \theta, \pi(e)\rangle - \log \psi(\theta)} \mathscr Z_{n-1,\theta}(\omega,e).
	\end{equation} On the other hand, if for $y \in \Z^d$ we define $T_y : \Omega \rightarrow \Omega$ to be the translation 
	\begin{equation} \label{eq:deft}
	T_y(\omega)(x):=\omega(x+y),
\end{equation} then it follows that for any $e \in \mathbb{V}$
	$$
	\mathscr Z_{n-1,\theta}(\omega,e)=\mathscr Z_{n-1,\theta}(T_e(\omega),0),
	$$ so that \eqref{eq:prelim1} becomes
	\begin{equation} \label{eq:prelim2}
	\mathscr Z_{n,\theta}(\omega,0)=\sum_{e \in \mathbb{V}(s)} \omega(0,e)\mathrm e^{\langle \theta, \pi(e)\rangle - \log \psi(\theta)} \mathscr Z_{n-1,\theta}(T_e(\omega),0).
	\end{equation} By translation invariance of $\P$ we know that $\mathscr Z_{n,\theta}(T_e(\omega),0) \rightarrow \mathscr Z_{\infty,\theta}(T_e(\omega),0)$ for $\P$-almost every $\omega$, so that we may take the $\P$-almost sure limit as $n \rightarrow \infty$ on \eqref{eq:prelim2} to obtain
	\begin{equation} \label{eq:prelim3}
	\mathscr Z_{\infty,\theta}(\omega,0)=\sum_{e \in \mathbb{V}(s)} \omega(0,e)\mathrm e^{\langle \theta, \pi(e)\rangle - \log \psi(\theta)} \mathscr Z_{\infty,\theta}(T_e(\omega),0).
	\end{equation} Moreover, it follows from \eqref{eq:prelim3} (and again translation invariance of $\P$) that the event $\{ \mathscr Z_{\infty,\theta}(0)=0\}$ is almost $T_e$-invariant for any $e \in \mathbb{V}(s)$ so that, by ergodicity of $\P$, its probability must be either $0$ or $1$. Since Lemma \ref{lema:1} dictates that the mean-one martingale $(\mathscr Z_{n,\theta}(0))_{n \in \N}$ converges to $\mathscr Z_{\infty,\theta}(0)$ in $L^2(\P)$, we have $\E(\mathscr Z_{\infty,\theta}(0))=1$ and thus it must be $\P(\mathscr Z_{\infty,\theta}(0)=0)=0$. By translation invariance of $\P$ we conclude the validity of the last sentence for all $x \in \Z^d$ so that
	$$
	\P\big\{\mathscr Z_{\infty,\theta}(x)=0 \text{ for some }x \in \Z^d\big\}=0,
	$$ implying the desired result. 
\end{proof}

\subsection{Concluding the proof of Theorem \ref{theo:2}.}\label{sec:conclude:theo2}

\subsubsection{Existence of the LDP limits and properties of moment generating functions.}

In order to conclude the proof of Theorem \ref{theo:2} we shall need Lemma \ref{lemma:3} below, which establishes the existence of certain ``point-to-point'' free energies (in the terminology of \cite{RS14}). Throughout the sequel, we will call a sequence $\{x_n\}_{n \in \N} \subseteq \Z^d$ \textit{admissible} if for each $n \in \N$ there exists a path $z=(z_0,z_1,\dots,z_n)$ of length $n$ with $z_0=0$ and $z_n=x_n$. 

\begin{lemma}\label{lemma:3} Under Assumption $\mathrm{A}$, for any $x\in \partial \mathbb{D}(s)$ there exists an admissible sequence $\{x_{n} \}_{n\in \N} \subseteq \Z^d$ such that $\tfrac{x_{n}}{n}\to x$ and
	\begin{align*}
		&\lim_{n\to \infty}\frac{1}{n}\log P_{0,\omega}(X_{n}=x_{n})= - I_q(x)\qquad\P\,\text{- a.s.},\\
		&\lim_{n\to \infty}\frac{1}{n}\log P_{0}(X_{n}=x_{n})=- I_a(x).
		\end{align*} 	
		\end{lemma}
	The proof of Lemma \ref{lemma:3} is deferred until the end of Section \ref{sec:conclude:theo2}. 

Next, recall from \eqref{Sn} that $S=(S_n)_{n \in \N}$ denotes the projected walk of the RWRE $X=(X_n)_{n\geq 0}$. Now, for each $n \geq 1$, let us set 
$$
\overline{S}_n:=\frac{1}{n}S_n
$$
 to be the empirical mean and, for each $n \geq 1$ and $\omega \in \Omega$, define the quenched log-moment generating function of $\overline{S}_n$ as 
$$
A^\omega_n(\theta):= \log E_{0,\omega} \Big( e^{\langle \theta , \overline{S}_n\rangle} \mathbbm{1}_{\mathbb B_n}\Big), \qquad \theta \in \R^{d-1},
$$ 
where the event $\mathbb B_n$ is defined in \eqref{eventBn}. Then the \textit{limiting} quenched log-moment generating function is 
\begin{equation} \label{lambda}
\Lambda^\omega(\theta)= \limsup_{n \rightarrow +\infty} \frac{1}{n} A^\omega_n(n\theta).
\end{equation}

We recall some qualitative properties of $\Lambda^\omega$ stated in the following result. 

\begin{lemma}\label{lemma:ge} For each $\bbP$ satisfying Assumption $\mathrm{A}$ there exists a full $\P$-probability event $\overline{\Omega}=\overline{\Omega}(\bbP)$ such that, for any $\omega \in \overline{\Omega}$, the following holds:
	\begin{itemize}
		\item [i.] The limit in \eqref{lambda} exists and is finite for all $\theta \in \R^{d-1}$, i.e. for all $\theta \in \R^{d-1}$
		$$
		\Lambda^\omega(\theta) = \lim_{n \rightarrow +\infty} \frac{1}{n}A^\omega_n(n\theta) \in (-\infty,+\infty).
		$$
		\item [ii.] $\Lambda^\omega$ is convex and continuous on $\R^{d-1}$.
		\item [iii.] If $y = \nabla \Lambda^\omega (\eta)$ for some $\eta \in \R^{d-1}$, then 
		$$
		\langle \eta , y \rangle - \Lambda^\omega (\eta)= \sup_{\theta \in \R^{d-1}} [ \langle \theta, y \rangle - \Lambda^\omega(\theta)]=:\overline{\Lambda}^\omega(y).
		$$ Moreover, $y$ is an exposed point of $\overline{\Lambda}^\omega$ and $\eta$ is its exposing hyperplane, i.e. for all $x \neq y$
		$$
		\langle \eta , y \rangle - \overline{\Lambda}^\omega(y) > \langle \eta,x\rangle - \overline{\Lambda}^\omega(x).
		$$ 
		\item [iv.] $\overline{\Lambda}^\omega$ is lower semicontinuous.
		\item [v.] For any closed set $F \subseteq \R^{d-1}$, 
		$$
		\limsup_{n \rightarrow +\infty} \frac{1}{n} \log P_{0,\omega}( \{\overline{S}_n \in F\} \cap \mathbb B_n ) \leq - \inf_{x \in F} \overline{\Lambda}^\omega(x).
		$$
		\item [vi.] For any open set $G \subseteq \R^{d-1}$, 
		$$
		\liminf_{n \rightarrow +\infty} \frac{1}{n} \log P_{0,\omega}( \{\overline{S}_n \in G\} \cap \mathbb B_n ) \geq - \inf_{x \in G \cap \mathcal{F}^\omega} \overline{\Lambda}^\omega(x),
		$$ where $\mathcal{F}^\omega$ denotes the set of exposed points of $\Lambda^\omega$. 
	\end{itemize}
\end{lemma}

\begin{proof} All the assertions are found in the standard literature (see \cite[Section 2.3]{DZ98}) which follows from the existence of a full $\P$-probability event $\overline{\Omega}$ such that, for any $\omega \in \overline{\Omega}$ and all $\theta \in \R^{d-1}$,
	\begin{equation} \label{eq:conv}
	\Lambda^\omega(\theta)=\lim_{n \rightarrow +\infty} \frac{1}{n} A^\omega_n(n\theta) < +\infty.
	\end{equation} Alternatively, once we have \eqref{eq:conv}, one can introduce the conditional probabilities 
	$$
	\mu_n:=P_{0,\omega}( \overline{S}_n \in \cdot\,|\, \mathbb B_{n})
	$$ and deduce the remaining parts of the lemma by applying the standard G\"artner-Ellis theorem for the sequence $(\mu_n)_{n \in \N}$.
	The existence of the limit \eqref{eq:conv} follows from \cite[Theorem 2.4-(b)]{RS14}, whereas its finiteness is a consequence of the simple bound $A_n^\omega(n\theta) \leq n|\theta|(d-1)$ for all $n$.
\end{proof}

\begin{remark}\label{remark:annealed}
As in the quenched set-up, we can define the annealed log-moment generating function
$$
A_n(\theta):= \log E_0 \Big( \mathrm e^{\langle \theta , \overline{S}_n\rangle} \mathbbm{1}_{\mathbb B_n}\Big) = n \log \psi(\tfrac{\theta}{n}),
$$ together with its limiting version
$$
\Lambda(\theta):= \lim_{n \rightarrow +\infty} \frac{1}{n}A_n(n\theta)= \log \psi(\theta).
$$
It is easy to see that an analogue of Lemma \ref{lemma:ge} holds for the annealed version $\Lambda$, by replacing $\Lambda^\omega$ with $\Lambda$ and $P_{0,\omega}$ with $P_0$ everywhere in the statements above. \qed
\end{remark}
	
\subsubsection{\bf Proof of Theorem \ref{theo:2}:}\label{sec:proof2}

We will now conclude the proof of Theorem \ref{theo:2} which will be carried out in a few steps. Throughout the following we assume $d \geq 4$ so that Proposition \ref{prop:pos} holds.

\noindent{\bf Step 1:} First, by Proposition \ref{prop:pos}, given any $\kappa > 0$ and $R>0$ there exists $\eps_R=\eps_R(d,\kappa,R)>0$ such that, for any $\bbP \in \mathcal{P}_\kappa$, whenever $\mathrm{dis}(\P)<\eps_R$ then, for each 
$$
\theta_0 \in D_R:=\{ \theta \in \R^{d-1} : |\theta| \leq R\},
$$
 we have that $\mathscr Z_{\infty,\theta_0}(0)$ is $\P$-a.s. strictly positive. Hence, it follows that for each $\bbP \in \mathcal{P}_\kappa$ there exists a full $\P$-probability event $\Omega_R=\Omega_R(\bbP,R)$ such that for all $\omega \in \Omega_R$
\begin{equation} \label{eq:spos}
\mathscr Z_{\infty,\theta}(\omega,0)>0 \text{ for all }\theta \in \Theta_R,
\end{equation} where $\Theta_R$ is some fixed (but arbitrary) countable dense subset of $D_R$. Furthermore, without loss of generality we may assume that $\Omega_R$ is contained in the event $\overline{\Omega}$ from Lemma \ref{lemma:ge}. But observe that, if this is the case, for $\omega \in \Omega_R$ and $\theta \in \Theta_R$ we may rewrite 
\begin{equation} \label{eq:eqlog}
\Lambda^\omega(\theta) = \log \psi(\theta) + \lim_{n \rightarrow +\infty} \frac{1}{n}\log \mathscr Z_{n,\theta}(\omega,0) = \log \psi(\theta),
\end{equation} where the second equality follows from \eqref{eq:spos}. Since $\Lambda^\omega$ is continuous on $D_R$ if $\omega \in \Omega_R$ by Lemma \ref{lemma:ge}, we conclude that for any such $\omega$ the equality $\Lambda^\omega(\theta)=\log \psi(\theta)$ in \eqref{eq:eqlog} holds for all $\theta$ in $D_R$. Therefore, we have shown that given any $\kappa,R>0$ there exists $\eps_R>0$ such that, for any $\bbP \in \mathcal{P}_\kappa$, whenever $\mathrm{dis}(\bbP)<\eps_R$ there exists a full $\P$-probability event $\Omega_R$ such that \eqref{eq:eqlog} holds for all $\theta \in D_R$ and $\omega \in \Omega_R$.

\smallskip

\noindent{\bf Step 2:} We now need the following result.

%\begin{defi} \label{def:good}
%We will say that a set $\mathcal O \subseteq \partial \mathbb{D}(s)$ is \textit{good} if there exists $R_{\mathcal O}>0$ such that
%	$$
%	\pi(\mathcal O) \subseteq \{ \nabla \log \psi(\theta) : \theta \in D_{R_{\mathcal O}}\},
%	$$ where $\pi$ is the affine transformation from $\eqref{eq:pi}$.
%\end{defi}

%The following lemma will be needed for the proof of Theorem \ref{theo:2}.

\begin{lemma}\label{lemma:good} Given $\kappa > 0$ and a compact set $\mathcal{K} \subseteq \partial \mathbb D(s) \setminus \partial\mathbb D_{d-2}$, there exists $R_\mathcal{K}=R_\mathcal{K}(d,\kappa,\mathcal{K})>0$ such that, for any $\bbP \in \mathcal{P}_\kappa$, we have 
	\[
	\pi(\mathcal K) \subseteq \{ \nabla \log \psi(\theta) : \theta \in D_{R_{\mathcal K}}\}.
	\]
\end{lemma}
We will assume Lemma \ref{lemma:good} for now and continue with the proof of Theorem \ref{theo:2}.

\smallskip

\noindent{\bf Step 3:} By Lemma \ref{lemma:good}, it will suffice to show that for any $\kappa,R > 0$ there exists $\eps=\eps(d,\kappa,R)>0$ such that, for any $\bbP \in \mathcal{P}_\kappa$, if $\mathrm{dis}(\mathbb P)<\eps$ then 
\begin{equation} \label{eq:eqono}
I_a\big|_{\mathcal O_R} \equiv I_q\big|_{\mathcal O_R}, 
\end{equation} where
\[
\mathcal{O}_R:= \pi^{-1} (\{ \nabla \log \psi(\theta) : \theta \in D_{R}\}).
\]To this end, let us consider $\varepsilon=\varepsilon_{R+1}>0$ depending only on $d,\kappa$ and $R$ such that, for any $\bbP \in \mathcal{P}_\kappa$, if $\mathrm{dis}(\mathbb P)<\eps$ there exists a full $\P$-probability event $\Omega_{R+1}=\Omega_{R+1}(\bbP,R)$ satisfying 
\begin{equation}
\label{eq:ratefunq}
\Lambda^\omega(\theta) = \log \psi(\theta)
\end{equation} for all $\theta \in D_{R+1}$ if $\omega \in \Omega_{R+1}$ (such an $\eps$ exists by Step 1). For the remaining steps of the proof, we fix an arbitrary $\bbP \in \mathcal{P}_\kappa$ satisfying $\mathrm{dis}(\bbP) < \eps$ and proceed to show \eqref{eq:eqono} for the RWRE having this environmental law $\bbP$. 

By \eqref{eq:ratefunq} and choice of $\eps$, it follows that 
\[
\pi(\mathcal O_R) = \{\nabla \Lambda^\omega(\theta) : \theta \in D_{R}\}
\] for any $\omega \in \Omega_{R+1}$. By Lemma \ref{lemma:ge}, it follows that for $\omega \in \Omega_{R+1}$ the sequence $(\overline{S}_n)_{n \in \N}$ under $P_{0,\omega}$ satisfies an LDP inside $\pi(\O_R)$ with rate function 
\begin{equation} \label{eq:rate}
\overline{\Lambda}(x)= \langle \theta_x , x \rangle - \log \psi(\theta_x),
\end{equation} where $\theta_x$ is defined via the relation $x=\nabla \log \psi(\theta_x)$ (observe that $\theta_x$ is well-defined for $x \in \pi(\O_R)$ by definition of $\O_R$). Here the LDP inside $\pi(\O_R)$ is interpreted as:
\begin{itemize}
	\item [$\bullet$] For any closed set $F \subseteq  \pi(\O_R)$, 
	$$
	\limsup_{n \rightarrow +\infty} \frac{1}{n} \log P_{0,\omega}( \{\overline{S}_n \in F\} \cap \mathbb B_n ) \leq - \inf_{x \in F} \overline{\Lambda}(x)
	$$
	\item [$\bullet$] For any open set $G \subseteq \pi(\O_R)$, 
	$$
	\liminf_{n \rightarrow +\infty} \frac{1}{n} \log P_{0,\omega}( \{\overline{S}_n \in G\} \cap \mathbb B_n ) \geq - \inf_{x \in G} \overline{\Lambda}(x),
	$$ 
\end{itemize} where $\overline{\Lambda}(x)$ is given by \eqref{eq:rate}. But an easy calculation exploiting the fact that $\pi$ is affine and \eqref{eq:equivalence0} shows that, for any set $H \subseteq \partial \mathbb{D}(s)$, we have 
\begin{equation} \label{eq:ident}
\{\bar{S}_n \in \pi(H)\}\cap \mathbb{B}_n = \{ \tfrac 1n X_n \in H\},
\end{equation} which implies then that an LDP inside $\mathcal{O}_R$ holds for the distribution of $(\tfrac 1n X_n)_{n \in \N_0}$ under $P_{0,\omega}$:
\begin{itemize}
	\item [$\bullet$] For any closed set $F \subseteq  \O_R$, 
	$$
	\limsup_{n \rightarrow +\infty} \frac{1}{n} \log P_{0,\omega}( \tfrac1n X_n \in F) \leq - \inf_{x \in F} \overline{\Lambda}(\pi(x))
	$$
	\item [$\bullet$] For any open set $G \subseteq \O_R$, 
	$$
	\liminf_{n \rightarrow +\infty} \frac{1}{n} \log P_{0,\omega}( \tfrac 1n X_n \in G ) \geq - \inf_{x \in G} \overline{\Lambda}(\pi(x)),
	$$ 
\end{itemize} where $\overline{\Lambda}$ is given by \eqref{eq:rate}. 

\smallskip

\noindent{\bf Step 4:} Our next step will be to show that $\overline{\Lambda} \circ \pi \equiv I_q$ on $\O_R$. To this end, suppose first that $\overline{\Lambda} (\pi(x)) < I_q(x)$ for some $x \in \O_R$. By the lower semicontinuity of $I_q$ we may find a neighborhood $B$ of $x$ such that $\inf_{y \in \overline{B}} I_q(y) > \overline{\Lambda} (\pi(x))$, where $\overline{B}$ denotes the closure of $B$. Observe that the set 
$$
G_x:=\pi( B ) \cap \{y \in \R^{d} : y_d = 0\}
$$ is an open set in $\R^{d-1}$. Thus, by Lemma \ref{lemma:ge} and \eqref{eq:ident}, for any $\omega \in \Omega_{R+1}$ we have
$$
\begin{aligned}
-\overline{\Lambda}(\pi(x)) = -\overline{\Lambda}^\omega(\pi(x)) &\leq - \inf_{y \in G_x \cap \mathcal{F}^\omega} \overline{\Lambda}^\omega(y) \\
&\leq \liminf_{n \rightarrow +\infty} \frac{1}{n} \log P_{0,\omega}(\tfrac 1n X_n \in \pi^{-1}(G_x)),
\end{aligned}
$$
and 
$$
\liminf_{n \rightarrow +\infty} \frac{1}{n} \log P_{0,\omega}(\tfrac 1n X_n \in \pi^{-1}(G_x))    
\leq \limsup_{n \rightarrow +\infty} \frac{1}{n} \log P_{0,\omega}(\tfrac 1n X_n \in \overline{B}) \leq - \inf_{y \in \overline{B}} I_q(y) < -\overline{\Lambda}(\pi(x)),
$$ which is a contradiction. Thus, we must have $I_q(x) \leq \overline{\Lambda}(\pi(x))$ for all $x \in \O_R$. 

On the other hand, if for each $x \in \O_R$ we choose an admissible sequence $\{x_n\}_{n \in \N}$ such that $\frac{x_n}{n} \to x$ as $n \rightarrow +\infty$ as in the statement of Lemma \ref{lemma:3}. Then, by the aforementioned lemma, \eqref{eq:ident} and Lemma \ref{lemma:ge}, for $\bbP$-almost every $\omega \in \Omega_{R+1}$ and $\delta > 0$ we have
$$
-I_q(x) = \lim_{n \rightarrow +\infty} \frac{1}{n} \log P_{0,\omega} ( \tfrac1n X_n = \tfrac 1n x_n) \leq \limsup_{n \rightarrow +\infty} \frac{1}{n} \log P_{0,\omega} ( \tfrac1n X_n \in \overline{B_\delta(x)}) \leq - \inf_{y \in \overline{B_\delta(x)}} \overline{\Lambda}^\omega(\pi(y)),
$$ with the standard notation $B_\delta(x):=\{y \in \R^d : |y-x| <\delta\}$. By the lower semicontinuity of $\overline{\Lambda}^\omega$, letting $\delta \to 0$ in the inequality above yields that 
$$
-I_q(x) \leq - \overline{\Lambda}^\omega(\pi(x)) = -\overline{\Lambda}(\pi(x)),
$$ the last equality being true by \eqref{eq:ratefunq} because $x \in \O_R$. Hence, we see that 
$$
\overline{\Lambda}(\pi(x)) \leq I_q(x)\qquad\forall \,x \in \O_R
$$ and therefore, since the reverse inequality is also true, we conclude that $I_q \equiv \overline{\Lambda} \circ \pi$ on $\O_R$.

\smallskip

\noindent{\bf Step 5:} Finally, a similar analysis but for the annealed measure  now reveals that $I_a \equiv \overline{\Lambda}\circ \pi$ on $\O_R$ as well. Indeed, the key observation to achieve this is that, by the analogue of Lemma \ref{lemma:ge} for the annealed measure (recall Remark \ref{remark:annealed}) the sequence $(\overline{S}_n)_{n \in \N}$ under $P_0$ satisfies an LDP inside $\pi(\O_R)$ with rate function exactly as in \eqref{eq:rate}. From here we immediately obtain \eqref{eq:eqono}. Thus, for the proof of Theorem \ref{theo:2}  we only owe the reader the proof of Lemma \ref{lemma:good} as well as Lemma \ref{lemma:3}. 

\smallskip 

\begin{proof}[{\bf Proof of Lemma \ref{lemma:good}:}] Fix any environment law $\bbP$ satisfying Assumption A with ellipticity constant $\kappa>0$. Since $\mathscr Z_{\theta}(0)$ is mean-one $\bbP$-martingale, it follows that
	\[
E_{0}(\e^{\langle \theta, S_1\rangle - \log\psi(\theta)} \mathbbm 1_{\mathbb B_n})=1.
\] From this identity, the methods from \cite[Section 4]{BMRS19} now show that the mapping $\theta \mapsto \log \psi(\theta)$ is smooth and has a positive definite Hessian. In particular, it is a smooth strictly convex function on $\R^{d-1}$, so that by \cite[Theorem 26.5]{R97} the sets 
$$
G_R:=\{ \nabla \log \psi(\theta) : |\theta| < R\}
$$ are open on $\R^{d-1}$ for all $R>0$. 

Therefore, in order to prove the lemma it will be enough to show that for each  $x \in \partial \mathbb D(s) \setminus \partial\mathbb D_{d-2}$ there exist $r_x=r_x(d,\kappa,x),R_x=R_x(d,\kappa,x)>0$ such that, for any $\bbP \in \mathcal{P}_\kappa$, we have 
\begin{equation} \label{eq:cover}
\pi (B_{r_x}(x) \cap ( \partial \mathbb D(s) \setminus \partial\mathbb D_{d-2})) \in G_{R_x} = \{ \nabla \log \psi(\theta) : |\theta| < R_x\},
\end{equation} where as usual we write $B_{r_x}(x):=\{ y \in \R^d : |y-x| < r_x\}$. Indeed, if this is the case then, given any compact set $\mathcal{K} \subseteq \partial \mathbb{D}(s) \setminus \partial \mathbb{D}_{d-2}$, there exists some finite $n_\mathcal{K}=n_\mathcal{K}(d,\kappa,\mathcal{K}) \geq 1$ and  $x_1,\dots,x_{n_{\mathcal{K}}} \in \mathcal{K}$ such that 
\[
\mathcal{K} \subseteq \bigcup_{j=1}^{n_{\mathcal{K}}} \big(B_{r_{x_j}}(x_j) \cap ( \partial \mathbb D(s) \setminus \partial\mathbb D_{d-2})\big),
\] so that by \eqref{eq:cover}, if we set $R_{\mathcal{K}}:=\max_{j=1,\dots,n_{\mathcal{K}}} r_{x_j} < \infty$ then,for any $\bbP \in \mathcal{P}_\kappa$, we obtain that
\[
\pi(\mathcal{K}) \subseteq G_{R_\mathcal{K}}.
\] Hence, we only need to show \eqref{eq:cover}.

To this end, notice that any $x \in \partial \mathbb{D}(s) \setminus \partial \mathbb{D}_{d-2}$ can be written as 
$$
x= \sum_{i=1}^d \delta_i s_i e_i
$$ where $\delta_i> 0$ for all $i=1,\dots,d$ and $\sum_{i=1} \delta_i = 1$. Since $\pi$ is affine, it follows that
$$
\pi(x)= \sum_{i=1}^{d} \delta_i \pi(s_i e_i)= \sum_{i=1}^{d-1} (\delta_i - \delta_d)e_i.
$$ On the other hand, a simple computation shows that for any $\theta \in \R^{d-1}$, 
$$
\nabla \log \psi(\theta) = \frac{1}{\psi(\theta)} \sum_{i=1}^{d-1} [\alpha(s_i e_i)\e^{\theta_i}-\alpha(s_de_d)\e^{-(\theta_1+\dots+\theta_{d-1})}]e_i.
$$ Therefore, in order to check that $\pi(x) \in G_R$ for some $R>0$, we only need to show that there exists some $\theta(x)=(\theta_1(x),\dots,\theta_{d-1}(x)) \in \R^{d-1}$ such that 
\begin{equation} \label{eq:igualdad}
\delta_i-\delta_d = \frac{1}{\psi(\theta)}\left[\alpha(s_i e_i)\e^{\theta_i(x)}-\alpha(s_de_d)\e^{-(\theta_1(x)+\dots+\theta_{d-1}(x))}\right]
\end{equation} for all $i=1,\dots,d-1$. But it is straightforward to check that, for $\theta(x)$ given by 
$$
\theta_i(x) :=  \log \left( \frac{\delta_i C}{\alpha(s_ie_i)} \right) \hspace{1cm}\text{ with }\hspace{1cm}C:= \sqrt[d]{\prod_{i=1}^d \frac{\alpha(s_i e_i)}{\delta_i}}
$$ for each $i=1,\dots,d-1$, \eqref{eq:igualdad} is satisfied and so $\pi(x)=\nabla \log \psi (\theta(x))$. Finally, since the mapping 
\[
(\alpha,x) \mapsto \theta(x)
\] is continuous on $\mathcal{M}^{(\kappa)}_1(\bbV) \times (\partial \mathbb{D}(s) \setminus \partial \mathbb{D}_{d-2})$, \eqref{eq:cover} follows upon taking $r_x:=\tfrac{1}{2}d(x,\partial \mathbb{D}_{d-2})>0$ and $R_x:=1+\sup \{ |\theta(y)| : y \in B_{r_x}(x) \cap (\partial \mathbb{D}(s) \setminus \partial \mathbb{D}_{d-2})\} < \infty$.
This concludes the proof. 
\end{proof}

\begin{remark} \label{rem:igualdad} Equation \eqref{eq:cover} implies that, for any $\bbP \in \mathcal{P}_\kappa$,
	\[
	\pi(\partial \mathbb{D}(s) \setminus \partial \mathbb{D}_{d-2}) \subseteq \{ \nabla \log \psi(\theta) : \theta \in \R^{d-1}\}.
	\] From this, we may now repeat the analysis of Step 5 to show that $I_a \equiv \overline{\Lambda} \circ \pi$ holds on $\partial \mathbb{D}(s)\setminus \partial \mathbb{D}_{d-2}$. We will later need this fact for the proof of Theorem \ref{theo:1}.%It follows from the analysis in Step 3  that under $P_0$ the sequence $(\tfrac{1}{n}X_n)_{n \in \N}$ satisfies an LDP with rate function $\overline{\Lambda} \circ \pi$ inside any $\mathcal{O} \subseteq \partial \mathbb{D}(s)\setminus \partial \mathbb{D}_{d-2}$ such that $\pi(\mathcal{O}) \subseteq \{\nabla \log \psi (\theta) : \theta \in \R^{d-1}\}$. By \eqref{eq:cover} and the discussion in Step 5, we conclude that 
\end{remark}

\subsubsection{\bf Proof of Lemma \ref{lemma:3}:} %We will carry out the proof of this lemma in several steps. 
We consider the quenched and annealed limits separately. 

\medskip
\textbf{Case 1: the quenched limit}. This is consequence of several results found in \cite{RS14}. Indeed, in \cite[Theorem 2.2]{RS14} it is proved that $\P$-almost surely for all $x\in \mathbb{D}$, the following limit exists \begin{equation*}
\widehat{I}_{q}(x):=-\lim_{n\to \infty}\frac{1}{n}\log P_{0,\omega}(X_{n}=x_{n})\in [0,\kappa]
\end{equation*}
for a suitable admissible sequence $\{x_{n}\}_{n \in \N}$ satisfying  $\tfrac{x_{n}}{n}\to x$. Moreover, by \cite [Theorem 2.4]{RS14} this limit $\widehat{I}_q(x)$ is deterministic and, by \cite[Theorem 3.2-(b)]{RS14}, the map $x\to \widehat{I}_{q}(x)$ is continuous on $\mathbb{D}$. Finally, \cite[Theorem 4.3]{RS14} shows that $I_{q}\equiv \widehat{I}_{q}$ on $\mathbb{D}^{\circ}$. The continuity of both $I_{q}$ and $\widehat{I}_{q}$ now allow us to extend the equality to the boundary $\partial \mathbb{D}$, thus proving the quenched case.

\medskip
\textbf{Case 2: the annealed limit}. First, given $x=(x_1,\dots,x_d) \in \partial \mathbb{D}(s)$, let us write it as $x=\sum_{i=1}^{d}s_{i}|x_{i}|e_{i}$. Now, consider any admissible sequence $\{x_n\}_{n \in \N} \subseteq \Z^d$ such that:
\begin{enumerate}
	\item [$\bullet$] $\tfrac{x_n}{n} \in \partial \mathbb{D}(s)$ for each $n$, i.e. $x_n= \sum_{i=1}^d s_in_ie_i$ for some $n_i \geq 0$ with $\sum_{i=1}^d n_i = n$.
	\item [$\bullet$] If $x_i=0$ then $n_i=0$.
	\item [$\bullet$] $\tfrac{x_n}{n} \to x$ as $n \to \infty$.
\end{enumerate} It is straightforward to check that such a sequence always exists, see \cite{RS14} for details.  

Observe that, for any such sequence, by \eqref{eq:cruobs} the quantity $\prod_{j=1}^{n}\alpha(\Delta_j(z))$ is independent of the path $z$ of length $n$ going from $0$ to $x_n$, so that 
\begin{equation*}
P_{0}(X_{n}=x_{n})=\#\{z \in \mathcal{R}_n: z_n=x_{n} \}\prod_{i=1}^{d}\alpha(s_ie_{i})^{n_{i}}=\frac{n!}{n_{1}!\cdots n_{d}!}\prod_{i=1}^{d}\alpha(s_ie_{i})^{n_{i}}
\end{equation*}
%Observe that if $x_{n}=\sum_{i=1}^{n}n_{i}e_{i}$ such that $\frac{x_{n}}{n}\to x$, then $\frac{n_{i}}{n}\to p_{i}$, so $n_{i}\approx np_{i}$. 	
Taking logarithm and dividing by $n$, we get \begin{equation*}
\frac{1}{n}\log P_{0}(X_{n}=x_{n})=\frac{1}{n}\left[\log n!-\sum_{i=1}^{d}\log n_{i}!+\sum_{i=1}^{d}n_{i}\log \alpha(s_ie_{i})\right]
\end{equation*}	
Now, since $\tfrac{x_n}{n} \to x$, we obtain that $\tfrac{n_i}{n} \to |x_i|$ for all $i$ and thus that as $n \to \infty$,
$$
\frac{1}{n}\sum_{i=1}^{d}n_{i}\log\alpha(s_ie_{i})\to \sum_{i=1}^{d}|x_{i}|\log \alpha(s_ie_{i}).
$$ On the other hand, since $\tfrac{n_i}{n} \to |x_i|$, by Stirling's approximation we have $\log n! = n\log n-n +o(n)$ and $\log n_i! = n_i \log n_i - n_i + o(n)$ (if $x_{i}=0$ for some $i$, the equality still holds since $n_i=0$), so that
\begin{align*}
\lim_{n \rightarrow \infty} \frac{1}{n}\left[\log(n!)-\sum_{i=1}^{d}\log(n_{i}!)\right]&= \lim_{n \rightarrow \infty} \frac{1}{n}\left[n\log n-n-\left(\sum_{i=1}^{d}n_i\log n_i - n_i \right)\right]\\
&=-\lim_{n \rightarrow \infty} \sum_{i=1}^d\frac{n_i}{n}\log\frac{n_i}{n}\\
&=-\sum_{i=1}^{d}|x_{i}|\log |x_{i}|.
\end{align*}	
Therefore, we conclude that
\begin{equation*}
-\widehat{I}_{a}(x):=\lim_{n\to \infty}\frac{1}{n}P_{0}(X_{n}=x_{n})=-\sum_{i=1}^{d}|x_i|\log \frac{|x_i|}{\alpha(s_ie_i)}.
\end{equation*}	To conclude the proof, we must now check that $\widehat{I}_a(x)=I_a(x)$. To this end, define 
\begin{equation} \label{eq:rst}
\widetilde{I}_a(x):= \sup_{\theta \in \R^d} (\langle \theta, x\rangle - \log \lambda(\theta))
\end{equation} for $\lambda$ as in \eqref{lambda_theta}. It is straightforward to check that $\widetilde{I}_a$ is the annealed rate function corresponding to a random walk $(Y_n)_{n \in \N_0}$ in a space-time random environment $\bar{\omega}=(\bar{\omega}(n,\cdot))_{n \in \N}$, where the $\bar{\omega}(n,\cdot)$ are i.i.d. having common law $\P$. Furthermore, by standard considerations of Fenchel-Legendre transforms (see Lemma \ref{lemma:ge}, for instance), it is straightforward to check that for all $x \in \partial \mathbb{D}(s)$ the supremum in \eqref{eq:rst} coincides with the expression derived for $\widehat{I}_a(x)$, so that $\widehat{I}_a(x)=\widetilde{I}_a(x)$. Thus, in order to conclude the proof, it will suffice to show that
\begin{equation} \label{eq:2ineq}
\widetilde{I}_a(x) \leq I_a(x) \leq \widehat{I}_a(x).
\end{equation} 

To check the right inequality in \eqref{eq:2ineq} we observe that, by the annealed LDP for the random walk and the fact that $\tfrac{x_n}{x} \to x$, for any $\delta>0$ we have
\begin{equation}\label{eq:2ineqa}
-\widehat{I}_a(x) \leq \limsup_{n \rightarrow \infty} \frac{1}{n} \log P_0( \tfrac{1}{n}X_n \in \overline{B_\delta(x)}) \leq -\inf_{y \in \overline{B_\delta(x)}	}I_a(y),
\end{equation} where $B_\delta(x):=\{y \in \R^d : |y-x| < \delta\}$. By the lower semicontinuity of $I_a$, taking $\delta \to 0$ in \eqref{eq:2ineqa} then yields the right inequality in \eqref{eq:2ineq}.

On the other hand, if $Q_0$ denotes the law of the random walk $(Y_n)_{n \in \N}$ in a space-time environment introduced previously starting from $0$, then for any $\delta > 0$ we have
\begin{equation} \label{eq:domeq}
P_0(\tfrac{1}{n}X_n \in B_\delta(x)) \leq \kappa^{-n\delta}Q_0(\tfrac{1}{n}Y_n \in B_\delta(x)).
\end{equation} Indeed, notice that 
\begin{align*}
P_0(\tfrac{1}{n}X_n \in B_\delta(x))&= \sum_{z \in \mathcal{R}_n \,:\, \frac{z_n}{n} \in B_\delta(x)} \E\left(\prod_{j=1}^n \omega(z_{j-1},\Delta_j(z))\right)\\ &\leq \kappa^{-n\delta}\sum_{z \in \mathcal{R}_n \,:\, \frac{z_n}{n} \in B_\delta(x)} \prod_{j=1}^n \alpha(\Delta_j(z)) = \kappa^{-n\delta} Q_0(\tfrac{1}{n}Y_n \in B_\delta(x)),
\end{align*} where the middle equality follows from the fact that the factors in the product $\prod_{j=1}^n \omega(z_{j-1},\Delta_j(z))$ are all independent except for at most $n\delta$ of them, but we can majorize these by independent versions at the expense of an additional $\kappa^{-1}$ factor. It follows from \eqref{eq:domeq} that 
$$
\inf_{y \in B_\delta(x)}I_a(y) \geq \inf_{y \in \overline{B_\delta(x)}} \widetilde{I}_a(y) - \delta\log k^{-1}. 
$$ By the lower semicontinuity of both $I_a$ and $\widetilde{I}_a$, letting $\delta \to 0$ in the last display above reveals that $\widetilde{I}_a(x) \leq I_a(x)$ and thus \eqref{eq:2ineq} is proved.
\qed

\begin{remark} In \cite[Theorem 4.3]{RS14} (see also \cite{CDRRS13}) it is shown that the sequence $(\tfrac{1}{n}X_n)_{n \in \N}$ satisfies a quenched LDP on $\partial \mathbb{D}$ with rate function $\widehat{I}_q$. Using this and Case 2 of Lemma \ref{lemma:3}, the analysis carried out in Section \ref{sec:proof2} (in particular, in Steps 4 and 5) already shows that for $\mathrm{dis}(\P)$ sufficiently small one has $\widehat{I}_q \equiv I_a$ on the boundary $\partial \mathbb{D}$. Some additional effort is required to show that $\widehat{I}_q \equiv I_q$ and thus conclude the result in Theorem \ref{theo:2}, but this is given by the other results from \cite{RS14} as shown in Case 1 of Lemma \ref{lemma:3}.
\end{remark}

\section{Proof of Theorem \ref{theo:1}}\label{sec:proof:theo:1}

{% \textcolor{red}{The argument has changed a bit, I would remove this paragraph
 Note that the proof of Theorem \ref{theo:2} in Section \ref{sec:proof:theo:2} already reveals that, in order to prove Theorem \ref{theo:1}, it suffices to prove that there exist $\eps^\star=\eps^\star(d,\kappa)>0$ such that, whenever $\mathrm{imb}_s(\bbP)$ is small enough, there exists some $\eta=\eta(d,\kappa)>0$ such that for each $\abs{\theta}\leq \eta$ we have 
        \begin{equation} \label{eq:l2}
        \sup_{n\ge 1}\|\mathscr Z_{n,\theta}\|^2_{L^2(\P)}< \infty.
        \end{equation}  
        The above estimate together with arguments similar to those given for the proof of Theorem \ref{theo:2} will then imply the desired equality of the rate functions on an open subset of $\partial \mathbb D(s) \setminus \partial\mathbb D_{d-2}$. 
        %Doing so this time will only yield a small enough open subset of $\partial \mathbb{D} \setminus \partial \mathbb{D}_{d-2}$ on which the equality of rate functions holds, because we will only have \eqref{eq:l2} on a small neighborhood of $\theta$.
        
        For $x \in \Z^d$, $e \in \bbV$ and $\theta \in \R^{d-1}$, define
        \[
        W(x,e,\theta):= \omega(x,e)\, \e^{\langle \theta,\pi(e)\rangle}
        \] and 
        \[
        W_s(x,\theta):=\sum_{e\in \mathbb{V}(s) } W(x,e,\theta) = \sum_{e \in \bbV(s)} \omega(x,e)\, \e^{\langle \theta,\pi(e)\rangle},
        \] where $\pi$ is the affine mapping from \eqref{eq:pi} and we use the identification $\pi(e) \in \R^{d-1}$ for $e \in \mathbb{V}(s)$. 
        Note that, since $\E(W_s(x,\theta))=\psi(\theta)$ for any $x \in \Z^d$ and, moreover, $\P$-almost surely for all $x \in \Z^d$ 
        \[
        W_s(x,\theta)\le \e^{(d-1)\abs{\theta}}\psi(0)(1+\mathrm{imb}_s(\bbP)) \leq \e^{2(d-1)\abs{\theta}}\psi(\theta)(1+\mathrm{imb}_s(\bbP)),
        \] we have (recall the definition of $\Delta_j(z)$ from \eqref{ez} and $\alpha^{\ssup\theta}$ from \eqref{qtheta}),        
        \begin{eqnarray}
        \label{gen1}
        \begin{split}
        &	\|\mathscr Z_{n,\theta}\|^2_{L^2(\P)}\\
        &\hspace{0.2cm}=\sum_{z,z' \in \mathcal{R}_n} \prod_{j=1}^n \bigg[\frac{
        \E\big( W(z_{j-1},\Delta_j(z),\theta) W(z'_{j-1},\Delta_j(z'),\theta)\big)}{\psi^2(\theta)}\bigg]\\
        &  \hspace{0.2cm}     =
        \sum_{z,z' \in \mathcal{R}_{n-1}} \frac{\mathbb E(W_s(z_{n-1},\theta)W_s(z'_{n-1},\theta))}{\psi^2(\theta)} \prod_{j=1}^{n-1} \bigg[\frac{
        	\E\big( W(z_{j-1},\Delta_j(z),\theta) W(z'_{j-1},\Delta_j(z'),\theta)\big)}{\psi^2(\theta)}\bigg]\\
        &\hspace{0.2cm}
        \le  \sum_{z,z' \in \mathcal{R}_{n-1}}     \e^{\mathscr V^{\ssup 0}_{s,\theta} \mathbbm
        	1_{\{z_{n-1} = z'_{n-1}\}}}        \prod_{j=1}^{n-1} \bigg[\frac{
        	\E\big( W(z_{j-1},\Delta_j(z),\theta) W(z'_{j-1},\Delta_j(z'),\theta)\big)}{\psi^2(\theta)}\bigg]\\
        &\hspace{0.2cm} =\sum_{z,z' \in \mathcal{R}_{n}}  
        \alpha^{\ssup\theta}(\Delta_n(z))\alpha^{\ssup\theta}(\Delta_n(z')) \e^{\mathscr V^{\ssup 0}_{s,\theta} \mathbbm
        	1_{\{z_{n-1} = z'_{n-1}\}}}    \prod_{j=1}^{n-1} \bigg[\frac{
        	\E\big( W(z_{j-1},\Delta_j(z),\theta) W(z'_{j-1},\Delta_j(z'),\theta)\big)}{\psi^2(\theta)}\bigg],
        \end{split}
        \end{eqnarray}
        where
        \[
        \mathscr V^{\ssup 0}_{s,\theta}:=2(d-1)\abs{\theta}+\log\left(1+\mathrm{imb}_s(\bbP)\right).
        \]
        We will now continue with an estimate for the sum over $z_{n-1}$ and
        $z'_{n-1}$. First, note that whenever $z_{n-2}\ne z'_{n-2}$ we have        
        \begin{equation}
        \label{en4}
        \begin{aligned}
       & \sum_{\Delta_{n-1}(z),\Delta_{n-1}(z') \in \mathbb{V}(s)}
        \e^{\mathscr V^{\ssup 0}_{s,\theta} \mathbbm 
        	1_{\{z_{n-1} = z'_{n-1}\}}}
        \bigg[\frac{
        	\E\big( W(z_{n-2},\Delta_{n-1}(z),\theta) W(z'_{n-2},\Delta_{n-1}(z'),\theta)\big)}{\psi^2(\theta)}\bigg]\\ 
        &= 
        \sum_{\Delta_{n-1}(z),\Delta_{n-1}(z') \in \mathbb{V}(s)}   \alpha^{\ssup\theta}(\Delta_{n-1}(z))\alpha^{\ssup\theta}(\Delta_{n-1}(z')) \e^{\mathscr V^{\ssup 0}_{s,\theta} \mathbbm 
        	1_{\{z_{n-1} = z'_{n-1}\}}}        . 
        \end{aligned}
        \end{equation}
        Next, we claim that if $\mathrm{imb}_s(\bbP) < (d-2)\kappa$ then $\mathbb P$-almost surely for all $x\in\mathbb Z^d$
        and $e\in \mathbb V(s)$,
        \begin{equation}
        \label{text1}
        \omega(x,e)\le (1-\kappa)\psi(0).
        \end{equation}
        Indeed, if \eqref{text1} is not satisfied for some
        $x'\in\mathbb Z^d$ and $e'\in \mathbb V(s)$ then, on a set of
        positive $\mathbb P$-measure we have that        
        $$
        \omega(x',e')> \psi(0)-(d-2)\kappa.
        $$
        Hence, by uniform
        ellipticity and the trivial bound $\psi(0)\leq 1$, we have on a set of positive $\mathbb P$-measure,
        $$
        W_s(x',0)=\omega(x',e')+\sum_{e'\ne e \in \mathbb{V}(s)} \omega(x',e)>(1-\kappa)\psi(0)+(d-1)\kappa \geq (1+(d-2)\kappa)\psi(0)
        $$
        which implies that $\mathrm{imb}_s(\bbP)> (d-2)\kappa$ and thus contradicts our assumptions. Hence, we conclude that, whenever $\mathrm{imb}_s(\bbP) < (d-2)\kappa$, \eqref{text1} holds and thus that $\P$-a.s. for all $x \in \Z^d$ and $e \in \mathbb{V}(s)$,
        \begin{equation} \label{eq:text2}
        \omega(x,e) \leq (1-\kappa)\psi(0) \leq \e^{(d-1)|\theta|}(1-\kappa)\psi(\theta).
        \end{equation}
       Now, whenever $z_{n-2}=z'_{n-2}$, using \eqref{eq:text2} we have
        \begin{equation}
        \label{en5}
        \begin{aligned}
        &\sum_{\Delta_{n-1}(z),\Delta_{n-1}(z') \in \mathbb{V}(s)} \e^{\mathscr V^{\ssup 0}_{s,\theta} \mathbbm 
        	1_{\{z_{n-1} = z'_{n-1}\}}}        
 \bigg[\frac{
	\E\big( W(z_{n-2},\Delta_{n-1}(z),\theta) W(z'_{n-2},\Delta_{n-1}(z'),\theta)\big)}{\psi^2(\theta)}\bigg]\\
        & \hspace{0.2cm}  =
        \sum_{\Delta_{n-1}(z),\Delta_{n-1}(z') \in \mathbb{V}(s)}
         \bigg[\frac{
        	\E\big( W(z_{n-2},\Delta_{n-1}(z),\theta) W(z_{n-2},\Delta_{n-1}(z'),\theta)\big)}{\psi^2(\theta)}\bigg]\\
        &\qquad+
        \sum_{\Delta_{n-1}(z),\Delta_{n-1}(z') \in \mathbb{V}(s)} \left[\e^{\mathscr V^{\ssup 0}_{s,\theta} \mathbbm 
        	1_{\{z_{n-1} = z'_{n-1}\}}}-1\right]	\bigg[\frac{\E\big( W(z_{n-2},\Delta_{n-1}(z),\theta) W(z_{n-2},\Delta_{n-1}(z'),\theta)\big)}{\psi^2(\theta)}\bigg]\\
      %  \frac{1}{\psi^2(\theta)}\\
        & \hspace{0.2cm} =         
        \frac{\E\left(W^2_s(z_{n-2},\theta)\right)}{\psi^2(\theta)}
        +
        \sum_{\Delta_{n-1}(z) \in \mathbb{V}(s)}\left( \e^{\mathscr V^{\ssup 0}_{s,\theta} }-1\right)\frac{
        \E\left(W^2(z_{n-2},\Delta_{n-1}(z),\theta)\right) }{\psi^2(\theta)}\\
        &\le
        \e^{\mathscr V^{\ssup 0}_{s,\theta} }
        +
        \big( \e^{\mathscr V^{\ssup 0}_{s,\theta} }-1\big)
        K_{\kappa,\theta} \\
        & = \sum_{\Delta_{n-1}(z),\Delta_{n-1}(z') \in \mathbb{V}(s)} \alpha^{\ssup \theta}(\Delta_{n-1}(z))\alpha^{\ssup \theta}(\Delta_{n-1}(z'))
        \e^{\mathscr V^{\ssup 1}_{\kappa,s,\theta} }\\
        &\le\sum_{\Delta_{n-1}(z),\Delta_{n-1}(z') \in \mathbb{V}(s)} \alpha^{\ssup \theta}(\Delta_{n-1}(z))\alpha^{\ssup \theta}(\Delta_{n-1}(z'))
        \e^{\mathscr V^{\ssup 1}_{\kappa,s,\theta} +\mathscr V^{\ssup 0}_{s,\theta}  1_{\{z_{n-1} = z'_{n-1}\}}},
        \end{aligned}
        \end{equation}
        where
        $$
        \begin{aligned}
        K_{\kappa,\theta}:=\e^{(d-1)\abs{\theta}}(1-\kappa)\qquad 
        \mbox{and}      \qquad  \e^{\mathscr V^{\ssup 1}_{\kappa,s,\theta} }:=
        \e^{\mathscr V^{\ssup 0}_{s,\theta} }
        +
        \big( \e^{\mathscr V^{\ssup 0}_{s,\theta} }-1\big)
        K_{\kappa,\theta}.
        \end{aligned}
        $$
        Combining (\ref{en4}) with (\ref{en5}) we see that
        \begin{equation}
        \begin{aligned}
        &\sum_{\Delta_{n-1}(z),\Delta_{n-1}(z') \in \mathbb{V}(s)}
        \e^{\mathscr V^{\ssup 0}_{\overline{\eps},\eps} \mathbbm 
        	1_{\{z_{n-1} = z'_{n-1}\}}} \bigg[\frac{
        	\E\big( W(z_{n-2},\Delta_{n-1}(z),\theta) W(z'_{n-2},\Delta_{n-1}(z'),\theta)\big)}{\psi^2(\theta)}\bigg]\\
        &\le \sum_{\Delta_{n-1}(z),\Delta_{n-1}(z') \in \mathbb{V}(s)} \alpha^{\ssup \theta}(\Delta_{n-1}(z))\alpha^{\ssup \theta}(\Delta_{n-1}(z'))
        \exp\bigg\{\mathscr V^{\ssup 1}_{\kappa,s,\theta}  1_{\{z_{n-2} = z'_{n-2}\}} +
        	\mathscr V^{\ssup 0}_{s,\theta} \mathbb 
        	1_{\{z_{n-1} = z'_{n-1}\}}\bigg\}.
        \end{aligned}
        \end{equation}
        From the above estimate and \eqref{gen1}, we conclude that
        \begin{equation}
        \begin{split}
        \|\mathscr Z_{n,\theta}&\|^2_{L^2(\P)}\le
         \sum_{z,z'\in\mathcal R_n} 
        \e^{\mathscr V^{\ssup 1}_{\kappa,s,\theta}  1_{\{z_{n-2}= z'_{n-2}\}} +
        	\mathscr V^{\ssup 0}_{s,\theta} \mathbb 
        	1_{\{z_{n-1}= z'_{n-1}\}}}\\
        &\hspace{2cm} \times   \prod_{j=1}^{n-2}\bigg[\frac{
        	\E\big( W(z_{j-1},\Delta_j(z),\theta) W(z'_{j-1},\Delta_j(z'),\theta)\big)}{\psi^2(\theta)}\bigg]
        \prod_{j=n-1}^n \alpha^{\ssup \theta}(\Delta_{j}(z))\alpha^{\ssup \theta}(\Delta_j(z^\prime)).
        \end{split}
        \end{equation}        
        By successive application of the above estimate, we get 
        $$
        \|\mathscr Z_{n,\theta}\|^2_{L^2(\P)}\le E^{\ssup\theta}_0\left[\exp\left(\sum_{j=0}^{n-1}\mathscr V^{\ssup{n-1-k}}_{\kappa,s,\theta}
        \mathbbm 1_{\{X_j^{\ssup \theta}=Y_j^{\ssup \theta}\}}\right)\right],
        $$
        where $X^{\ssup \theta}$ and $Y^{\ssup \theta}$ are as before two independent random walks starting from $0$ with jump distribution given by the probability vector $\vec{\alpha}^{\ssup \theta}$, we write $\mathscr V^{\ssup 0}_{\kappa,s,\theta}:=\mathscr V^{\ssup 0}_{\overline{\eps},\theta}$ for homogeneity of notation and, for $0\le k\le n-1$, we define
        $$
        \e^{\mathscr V^{\ssup{k+1}}_{\kappa,s,\theta}}:=\e^{\mathscr V^{\ssup 0}_{\kappa,s,\theta}}
        +\left( \e^{\mathscr V^{\ssup k}_{\kappa,s,\theta} }-1\right)
        K_{\kappa,\theta}.
        $$
        Now, since $K_{\kappa,\theta}<1$ for $|\theta|$ small enough (depending only on $d$ and $\kappa$), for any such $\theta$ we have       
        $$
        \e^{\mathscr V^{\ssup{k+1}}_{\kappa,s,\theta}}=\e^{\mathscr V^{\ssup 0}_{\kappa,s,\theta}}
        +\left( \e^{\mathscr V^{\ssup 0}_{\kappa,s,\theta}}-1\right)
        (K_{\kappa,\theta}+K^{2}_{\kappa,\theta}+\cdots+K^{k+1}_{\kappa,\theta})
        \le
       \e^{\mathscr V^{\ssup 0}_{\kappa,s,\theta}}
        +\left( e^{\mathscr V^{\ssup 0}_{\kappa,s,\theta}}-1\right)\frac{K_{\kappa,\theta}}{1-K_{\kappa,\theta}}.
        $$
        Hence, we can define $\mathscr V^{\ssup \infty}_{\kappa,s,\theta}$ by the formula
        $$
        \e^{\mathscr V^{\ssup \infty}_{\kappa,s,\theta}}:=
        \e^{\mathscr V^{\ssup 0}_{\kappa,s,\theta}}
        +\left( \e^{\mathscr V^{\ssup 0}_{\kappa,s,\theta}}-1\right)\frac{K_{\kappa,\theta}}{1-K_{\kappa,\theta}},
        $$
        and conclude that
        \begin{equation}
        \label{gn2}
        \|\mathscr Z_{n,\theta}\|^2_{L^2(\P)}\le E^{\ssup \theta}_0\left[\exp\left(\mathscr V^{\ssup\infty}_{\kappa,s,\theta}\sum_{j=0}^{n-1} 
        \mathbbm 1_{\{Z_j^{(\theta)}=0\}}\right)\right]
        \end{equation} where $Z_j^{\ssup\theta}=X_j^{\ssup \theta}-Y_j^{\ssup\theta}$. Moreover, since for $\abs{\theta}\leq \eta_1(d,\kappa)$ we have $K_{\kappa,\theta} \leq 1-\frac{\kappa}{2}$ and $2(d-1)|\theta|\leq 1$, a straightforward calculation yields that
        \begin{equation} \label{eq:boundvinf}
        \mathscr{V}^{\ssup \infty}_{\kappa,s,\theta} \leq C_1 (|\theta| + \mathrm{imb}_s(\bbP))(1+\mathrm{imb}_s(\bbP))
        \end{equation} for some constant $C_1=C_1(d,\kappa)>0$.  
        
        Now, by \eqref{eq:boundsumtheta} there exists $C_0=C_0(d,\kappa)>0$ such that 
        \begin{equation}
        \label{eq:boundsum}
        \sup_{|\theta| \leq 1} \sum_{j=0}^\infty P_0\big( Z^{\ssup\theta}_j=0\big) \leq C_0.
        \end{equation}
        It then follows from \eqref{eq:boundvinf} that there exist $\eta_2=\eta_2(d,\kappa) \in (0,\eta_1)$ and $\eps'=\eps'(d,\kappa) > 0$ such that, for any $\bbP \in \mathcal{P}_\kappa$, if $\mathrm{imb}_s(\bbP) < \eps'$ then $\sup_{|\theta|< \eta_2} \mathscr{V}^{\ssup \infty}_{\kappa,s,\theta} < C_0^{-1}$ which, by Lemma \ref{lema:2} and \eqref{eq:boundsum}, implies 
        \[
        \sup_{|\theta| < \eta_2\,,\,n \geq 0} \|\mathscr Z_{n,\theta}\|^2_{L^2(\P)} < \infty.
        \] 
        
        The rest of the proof of \eqref{eq:equality1} now follows the same line of arguments as that of Theorem \ref{theo:2}. In the end, we obtain that there exist $\eta=\eta(d,\kappa),\eps^*=\eps^\star(d,\kappa)>0$ such that, for any $\bbP \in \mathcal{P}_\kappa$, if $\mathrm{imb}_s(\bbP) < \eps^\star$ then $I_q(x)=I_a(x)$ for all $x \in \mathcal{O}$, where $\mathcal{O} \subseteq \partial \mathbb{D}(s) \setminus \partial \mathbb{D}_{d-2}$ is the open set given by
        \begin{equation} \label{eq:seto}
        \mathcal{O}:=\pi^{-1}(\{\nabla \log \psi(\theta) : |\theta| < \eta\}) \cap (\partial \mathbb{D}(s) \setminus \partial \mathbb{D}_{d-2}).
        \end{equation} Since the mapping $(\alpha,\theta) \mapsto \mathrm{Hessian}(\log \psi(\theta))$ is continuous on $\mathcal{M}^{(\kappa)}(\bbV) \times \R^{d-1}$, by \cite[Theorem 4.5]{BMRS19} (see also the proof of \cite[Lemma 4.8]{BMRS19}) there exists $r=r(d,\kappa)>0$ such that, for any $\bbP \in \mathcal{P}_\kappa$,
        \[
        B_r(\nabla \log \psi (0)) \subseteq  \{\nabla \log \psi(\theta) : |\theta| < \eta\}.
        \] From this, standard properties of affine transformations show that there exists some $c > 0$ depending only on the transformation $\pi$ such that
        \[
        B_{cr}(\overline{x}_s)  \cap (\partial \mathbb{D}(s) \setminus \partial \mathbb{D}_{d-2}) \subseteq \mathcal{O}, 
        \] for $\overline{x}_s$ defined as 
        \begin{equation}
        \label{eq:defs}
        \overline{x}_s:=\pi^{-1}(\nabla \log \psi(0)) = \frac{1}{\psi(0)}\sum_{i=1}^d \alpha(s_ie_i)s_ie_i \in \partial \mathbb{D}(s) \setminus \partial \mathbb{D}_{d-2}.
        \end{equation}
 
        Finally, to check \eqref{equality-infimum} we first observe that $I_a \leq I_q$ by Jensen's inequality and Fatou's lemma (or Lemma \ref{lemma:3}), so that it will be enough to show that $I_q(x_0)=\min_{x \in \partial \mathbb{D}(s)}I_a(x)$ for some $x_0 \in \partial \mathbb{D}(s)$. Now, by Lemma \ref{lemma:ge} and Remarks \ref{remark:annealed}-\ref{rem:igualdad} we have that $\min_{x \in \partial \mathbb{D}(s)}I_a(x)=I_a(\overline{x}_s)$ for $\overline{x}_s$ as in \eqref{eq:defs}. Since $\overline{x}_s$ belongs to the set $\mathcal{O}$ in \eqref{eq:seto}, we see that $I_q(\overline{x}_s)=I_a(\overline{x}_s)=\min_{x \in \partial \mathbb{D}(s)}I_a(x)$ and so \eqref{equality-infimum} now follows.

\section{Proofs of Theorem \ref{theo:4}-Theorem \ref{theo:5}}\label{sec:proof:theo:5}

\noindent{\bf Proof of Theorem \ref{theo:4}.} That $I_a(x)=\sup_{\theta\in\R^{d}}[\langle\theta,x\rangle- \log\lambda(\theta)]$ has been shown already in Case~2 of the proof of Lemma \ref{lemma:3}. Theorem \ref{theo:1}-Theorem \ref{theo:2} then imply the desired identity for $I_q$. 
\qed

\smallskip

\noindent{\bf Proof of Theorem \ref{theo:5}.} Recall that in this context the environments admit the representation
\begin{equation*}
\omega_\eps(x,e):=\alpha(e)(1+\eps\eta(x,e)),
\end{equation*}
for $\eps\in [0,1)$ and $\{\eta(x,\cdot) \}_{x\in \Z^{d}}$ an i.i.d. family of mean-zero random vectors on $\Gamma_\alpha$. To emphasize the dependence on the disorder parameter, we henceforth write $I_{q}(\cdot,\eps)$ and $I_{a}(\cdot,\eps)$ respectively for the quenched and annealed large deviation rate functions of the random walk in the environment~$\omega_{\eps}$.
For $x \in \partial\mathbb{D}\setminus \partial \mathbb{D}_{d-2}$ define 
\begin{equation} \label{eq:defec}
\eps_c:=\sup \{ \eps \in [0,1): I_q(x,\eps) = I_a(x,\eps)\}.
\end{equation} Note that we always have $I_{q}(\cdot,0)\equiv I_{a}(\cdot,0)$ since $\omega_0$ is non-random, so that the set in \eqref{eq:defec} is always nonempty. Furthermore, by Theorem \ref{theo:2} we have that $I_q(x,\eps)=I_a(x,\eps)$ for all $\eps$ sufficiently small, so that in fact $\eps_c(x)>0$ for all $x \in \partial\mathbb{D}\setminus \partial \mathbb{D}_{d-2}$. % (one can check that in fact $\eps_c(x)\geq \tfrac{1}{2}\alpha^*\eps_{\{x\}}$, where $\alpha^*:=\min_{e \in \mathbb{V}} \alpha(e)$ and $\eps_{\{x\}}$ is the $\eps_{\mathcal{K}}$ from Theorem \ref{theo:2} for $\mathcal{K}:=\{x\}$).
Assuming that the mapping $\eps\mapsto I_{a}(\cdot,\eps)-I_{q}(\cdot,\eps)$ is monotone for the moment, let us deduce \eqref{truephasetransition}.	

\smallskip

\noindent{\bf Proof of \eqref{truephasetransition}:} Choose any probability measure $\mathbb{Q}$ satisfying Assumption $\mathrm{B}$ and $\eps'\in(0,1)$. By \cite[Proposition 4]{Y11}, $I_{a}(x_0,\eps')<I_{q}(x_0,\eps')$ for some $x_0\in \partial \mathbb{D}$.\footnote{Even though \cite[Proposition 4]{Y11} states that the strict inequality holds for some interior point $x_0 \in \mathbb{D}^\circ$, the proof actually shows that the inequality holds for some $x_0 \in \mathbb{V} \subseteq \partial \mathbb{D}$.} As the rate functions are continuous on $\mathbb{D}$, there exists an open set $\mathcal O\subset  \partial \mathbb{D}\setminus\partial \mathbb{D}_{d-2}$ on which the inequality above holds. Since $\varepsilon_c(x)>0$ for all $x \in \partial\mathbb{D}\setminus \partial \mathbb{D}_{d-2}$ by Theorem \ref{theo:2}, the monotonicity of the map $\eps \mapsto I_a(x,\eps)-I_q(x,\eps)$ now implies that $0<\eps_{c}(x)\leq \eps'$ for all $x \in \mathcal{O}$ which, since $\eps' <1$, shows \eqref{truephasetransition} and therefore proves the existence of a true phase transition. \qed

\smallskip	
	
For the proof of Theorem \ref{theo:5}, we now owe the reader the proof of \eqref{phasetransition}. 

\smallskip

\noindent{\bf Proof of \eqref{phasetransition}:} By the uniform ellipticity of $\omega_{\eps}$, the proof of this part now follows from Lemma~\ref{lemma:3}, the dominated convergence theorem, and 
	\begin{lemma}\label{lemma:2} 
	Fix $x \in \partial \mathbb{D}$ and let $\{x_n\}_{n \in \N} \subseteq \Z^d$ be the corresponding admissible sequence from Lemma \ref{lemma:3}. Then, under Assumption $\mathrm{B}$, for all $n \in \N$ the map 
	\begin{equation*}
	\eps\mapsto \frac{1}{n}\Big[\E\log P_{0,\omega_{\eps}}(X_{n}=x_{n})-\log P_{0}(X_{n}=x_{n}) \Big]=:D_n(\eps)
	\end{equation*}
	is non-increasing. Moreover, the map $\eps \mapsto I_a(x,\eps)-I_q(x,\eps)=\lim_{n \to \infty} D_n(\eps)$ is continuous on~$[0,1)$.
	\end{lemma}
\begin{proof}[\bf Proof of Lemma \ref{lemma:2}]
%The proof of Lemma \ref{lemma:2} is consequence of the Harris-FGK(\cite{H60}) and it is an adaptation of the arguments in Theorem 2.1 of \cite{C17}.\\
Fix $n\in \N$ and $x_n \in \Z^d$ with $|x_n|=n$. Then, in the notation of \mbox{Section \ref{sec:proj},} by \eqref{eq:obs} and \eqref{eq:cruobs} we can compute explicitly 
\begin{align*}
P_{0,\omega_{\eps}}(X_{n}=x_{n})&=\sum_{z=(z_0,\dots,z_{n-1}) \in \mathcal{R}^{(x_n)}_{n-1}}\prod_{j=1}^{n}\left(\alpha(\Delta_j(z))(1+\eps\eta(z_{j-1},\Delta_j(z)))\right)\\
P_{0}(X_{n}=x_{n})&=\sum_{z=(z_0,\dots,z_{n-1}) \in \mathcal{R}^{(x_n)}_{n-1}}\prod_{j=1}^{n}\alpha(\Delta_j(z)),
\end{align*} where $\mathcal{R}_{n-1}^{(x_n)}$ is the set of paths $z$ of length $n-1$ which start at $0$ and end at some neighbor of $x_n$, i.e. all paths $z \in \mathcal{R}_{n-1}$ such that $\Delta_n(z):=x_n - z_{n-1} \in \mathbb{V}$. 

To show that $D_n$ is non-increasing, it will be enough to show that its derivative $\frac{\d}{\d\eps}$ is non-positive. The second term in $D_n$ does not depend on $\eps$, so by uniform ellipticity we have for $\eps \in (0,1)$ 
\begin{align}
\frac{\mathrm d D_n}{\mathrm d\eps}=\frac{1}{n}\frac{\mathrm d}{\mathrm d\eps}\bigg(\mathbb{E}\big[\log P_{0,\omega_{\eps}}(X_{n}=x_{n})\big]\bigg)&=\frac{1}{n}\mathbb{E}\bigg[\frac{\mathrm d}{\mathrm d\eps}\bigg(\log P_{0,\omega_{\eps}}(X_{n}=x_{n})\bigg)\bigg] \nonumber\\
&=\frac{1}{n}\sum_{z \in \mathcal{R}^{(x_n)}_{n-1}}\E\left[\frac{A_n(z)B_n(z)}{\sum_{z' \in \mathcal{R}_{n-1}^{(x_n)}}A_n(z')}\right]\label{eq:form},
\end{align}
where 
\begin{equation}\label{eq:deriv}
A_n(z):=\prod_{j=1}^{n}\left(\alpha(\Delta_j(z))(1+\eps\eta(z_{j-1},\Delta_j(z)))\right) \hspace{0.5cm}\text{ and }\hspace{0.5cm}
B_n(z):=\sum_{j=1}^{n}\frac{\eta(z_{j-1},\Delta_j(z))}{1+\eps\eta(z_{j-1},\Delta_j(z))}.
\end{equation}
Next, for each path $z \in \mathcal{R}_{n-1}^{(x_n)}$ let us define the probability measure $P^z$ given by 
\begin{equation*}
\mathrm d P^{z}=\frac{A_n(z)}{\prod_{j=1}^{n}\alpha(\Delta_j(z))}\d\mathbb{Q}.
\end{equation*}
Recalling \eqref{eq:deriv}, this allows us to write the derivative as 
\begin{equation*}
\frac{\d D_n}{\d\eps} = \frac{1}{n}\sum_{z \in \mathcal{R}_{n-1}^{(x_n)}}\left[\prod_{j=1}^{n}\alpha(\Delta_j(z))\right]E^{z}\left[\frac{B_n(z)}{\sum_{z' \in \mathcal{R}_{n-1}^{(x_n)}}A_n(z')} \right].
\end{equation*}
Note that, for each $z \in \mathcal{R}_{n-1}^{(x_n)}$, the random variables $(\eta(z_{j-1},\Delta_j(z)) : j=1,\dots,n)$ are independent under $P^{z}$ (although not necessarily identically distributed). Furthermore, observe that $A_n(z)$ and $B_n(z)$ are both increasing in $\eta$ for any path $z$. Therefore, by uniform ellipticity and the 
Harris-FKG inequality (see \cite{H60}) we conclude that for any $\eps \in (0,1)$,
\begin{align*}
\frac{\d D_n}{\d\eps}&\leq \frac{1}{n} \sum_{z \in \mathcal{R}_{n-1}^{(x_n)}}\left[\prod_{j=1}^{n}\alpha(\Delta_j(z))\right]E^{z}(B_n(z))E^{z}\left[\frac{1}{\sum_{z' \in \mathcal{R}_n^{(x_n)}}A_n(z')}\right]\\
&=\frac{1}{n}\sum_{z \in \mathcal{R}_{n-1}^{(x_n)}}\left(\prod_{j=1}^{n}\alpha(\Delta_j(z))\right)^{-1}\E(A_n(z)B_n(z))\E\left[\frac{A_n(z)}{\sum_{z' \in \mathcal{R}_{n-1}^{(x_n)}}A_n(z')}\right]=0,
\end{align*}
where the last equality follows from the fact that $\E(A_n(z)B_n(z))=0$ since the random variables $(\eta(z_{j-1},\Delta_j(z)) : j=1,\dots,n)$ all have mean zero and are independent under $\Q$ by \eqref{eq:cruobs}. Thus, we see that $\frac{\d D_n}{\d\eps} \leq 0$ and therefore $D_n$ is non-increasing on $[0,1)$. Finally, to show that the map $D_\infty(\eps):=I_a(x,\eps)-I_q(x,\eps)$ is continuous we first observe that for any $\eps' \in (0,1)$ there exists some $C_{\eps'}>0$ such that  $\sup_{\eps \leq \eps'}|B_n(z)|\leq C_{\eps'}n$ for all paths $z \in \mathcal{R}_{n-1}^{(x_n)}$. By \eqref{eq:form}, this implies that $\sup_{\eps \leq \eps'} \frac{\d D_n}{\d\eps}(\eps)\leq C_{\eps'}$ for any $\eps' \in (0,1)$, and so by the mean value theorem 
\begin{equation}\label{eq:lip}
|D_n(\eps_1)-D_n(\eps_2)|\leq C_{\eps'}|\eps_1-\eps_2|
\end{equation} for any $\eps_1,\eps_2 \in [0,\eps']$. Since $\eps'$ can be taken arbitrarily close to $1$, the continuity of $D_\infty$ now follows upon taking the limit as $n \rightarrow \infty$ on \eqref{eq:lip}, since $D_\infty(\epsilon)=\lim_{n \to \infty} D_n(\eps)$ by Lemma \ref{lemma:3}.
\end{proof}

\noindent{\bf Acknowledgement.} 
The authors would like to thank Noam Berger (Munich), Nina Gantert (Munich) and Atilla Yilmaz (Philadelphia) for very helpful comments on an earlier version of this manuscript. 
R. Bazaes has been supported by ANID-PFCHA/Doctorado Nacional no. 2018-21180873. A. F. Ram\'irez has been partially supported by Fondo Nacional
de Desarrollo  Cient\'ifico y Tecnol\'ogico 1180259 and Iniciativa
Cient\'\i fica Milenio.
Research of C. Mukherjee is supported by the Deutsche Forschungsgemeinschaft (DFG) under Germany's Excellence Strategy EXC 2044--390685587, Mathematics M\"unster: Dynamics--Geometry--Structure.  S. Saglietti has been supported in part at the Technion by a fellowship from the Lady Davis Foundation, the Israeli Science Foundation grants no. 1723/14 and 765/18, and by the NYU-ECNU Institute of Mathematical Sciences at NYU Shanghai. This research was also supported by a grant from the United States-Israel Binational Science Foundation (BSF), no. 2018330.

\end{document}